\documentclass[]{article}

\usepackage[utf8]{inputenc}
\usepackage[caption=false]{subfig}

\usepackage[numbers,sort&compress]{natbib}
\usepackage{graphics,graphicx,epsfig,float}
\usepackage{amscd}
\usepackage{anyfontsize}
\usepackage{amsmath}
\usepackage{amsthm}
\usepackage{amssymb}
\usepackage{authblk}
\bibpunct[, ]{[}{]}{,}{n}{,}{,}

\theoremstyle{plain}
\newtheorem{theorem}{Theorem}[section]
\newtheorem{lemma}[theorem]{Lemma}
\newtheorem{corollary}[theorem]{Corollary}
\newtheorem{proposition}[theorem]{Proposition}
\newtheorem{remark}[theorem]{Remark}

\theoremstyle{definition}
\newtheorem{definition}[theorem]{Definition}
\newtheorem{example}[theorem]{Example}

\DeclareMathOperator{\diam}{diam}
\DeclareMathOperator{\id}{id}
\DeclareMathOperator{\Int}{Int}

\begin{document}

\title{Chaotic behavior of countable products of homeomorphism groups}


\author[1]{N.~I. Zhukova\thanks{CONTACT: N.~I. Zhukova. Email: nina.i.zhukova@yandex.ru, nzhukova@hse.ru}}
\author[2]{A.~G. Korotkov}
\affil[1]{HSE University,

ul. Bolshaja Pecherskaja, 25/12, 603155 Nizhny Novgorod, Russia}
\affil[2]{Lobachevsky State University of Nizhny Novgorod, Gagarin ave., 23, build. 2, 603022 Nizhny Novgorod, Russia}
\date{}

\maketitle

\begin{abstract} Relationships between a chaotic behavior and closely related properties of topological 
transitivity, sensitivity to initial conditions, density of closed orbits of homeomorphism groups and 
their countable products are investigated. We construct numerous new examples of chaotic groups of homeomorphisms 
of countable products of various metrizable topological spaces, including infinite-dimensional
topological manifolds, whose factors can be as noncompact surfaces, so triangulable closed manifolds 
of an arbitrary dimension.
\end{abstract}


\vskip10mm
\section{Introduction}
\noindent

In theory of chaos, the chaotic behavior of infinite products of transformation groups has not been investigated.
The reason is that, by definition, the chaotic behavior of a group of homeomorphisms assumes the density 
of the set of finite orbits (see, for example, \cite{CarDEKP}, \cite{CarK}, \cite{NSan}). This requirement may not be  
fulfilled when moving to the infinite product of spaces. Following \cite{BGZ}, we give a more 
general definition of the chaotic behavior of groups of homeomorphisms by weakening the above condition by 
requiring the density of the set of closed orbits (Definition~\ref{Chaos2}). Here by a closed orbit we mean an orbit which
is a closed subset of the respective topological space. This allows us to investigate the chaotic behavior of arbitrary infinite
products of homeomorphism groups. 

\subsection{Devaney's chaos}
Let $T: X\to X$ be a continuous map of metric space. The family $\{T^n\}_{n \in \mathbb N}$ denoted by the pair 
$(X, T)$ is called a dynamical system. Devaney \cite{Devaney} proposed the following notion of chaos, which is usually 
called Devaney's chaos.

\begin{definition}\label{Dev} A dynamical system $(X,T)$ is called {\it chaotic in the sense of Devaney} if it satisfies the
following three properties:
\begin{enumerate}
	\item[(1)] $(X,T)$ is topologically transitive;
  \item[(2)] the set of periodic points of $(X,T)$ is dense in $X$;
  \item[(3)] $(X,T)$ has sensitive dependence on initial conditions.
\end{enumerate}
\end{definition}

Sensitive dependence on initial conditions is widely understood as being the central idea of chaos. In \cite{BBC} it was
shown that in Devaney's definition of chaos, the sensitive dependence follows from transitivity and density of periodic orbits. 
It was found in \cite{Assaf} that neither transitivity nor density of periodic trajectories are deducible from the remaining 
two conditions.

In \cite{CarDEKP} G. Cairns, G. Davis, D. Elton, A. Kolganova and P. Perversi introduced the following notion of a chaotic group 
action as a generalization of chaotic dynamical systems (Definition~\ref{Chaos1}). They showed that, if a group $G$ acts chaotically 
on a compact Hausdorff space, then $G$ is residually finite. Moreover, the reverse is also true, i.e. for every residually finite group 
$G$ there exists a Hausdorff space on which $G$ acts chaotically. As in \cite{CarDEKP}, we don't
assume any topology on the group $G,$ but we assume that each element of $G$ acts on a topological space $X$ as a 
homeomorphism of $X$, and the set $X$ is infinite. All group actions are assumed to be faithful, 
i.e. the only element of a group $G$ which acts as identity homeomorphism is the neutral element in $G$.

\begin{definition}\label{Chaos1} A group of homeomorphisms $G$ of a Hausdorff topological space $X$ is called 
chaotic if the following two conditions are met:
\begin{enumerate}
\item[(1)] {\it topological transitivity}: for every pair of nonempty open subsets $U$ and $V$ in $X$, there exists an element $g\in G$
such that $g(U)\cap V\neq\varnothing$;

\item[(2)] {\it finite orbits dense}: the union of finite orbits is dense in $X$.
\end{enumerate}
\end{definition}

Following \cite{BGZ}, we give and use in this work a different definition of the chaotic behavior of an arbitrary 
homeomorphism group $G$.

\begin{definition}\label{Chaos2} A group of homeomorphisms $G$ of a topological space $X$ is called {\it chaotic} 
(or $G$ has a chaotic behavior) on $X$ if the following two conditions are met:
\begin{enumerate}
\item[(1)] there exists a dense non-closed orbit of the group $G$ in $X$ ({\it the existence of a dense orbit});

\item[(2)] the union of closed orbits is dense in $X$ ({\it the density of closed orbits}).
\end{enumerate}
\end{definition}

Note that Definition~\ref{Chaos2} is more general than Definition~\ref{Chaos1} in the class of $T_1$-spaces.
Emphasize that in the case when $G$ is a countable homeomorphism group of a metrizable
compact space $X$, Definitions~\ref{Chaos1} and \ref{Chaos2} are equivalent (Proposition~\ref{Birkhoff} and Lemma~\ref{LNon}).

If $(X,d)$ is a metric space, then we define the notion of a sensitive dependence of a homeomorphism group $G$ on 
initial conditions (in Section~\ref{S6}).

\subsection{The organization of this work. Main results}

For the convenience of the reader, we provide the basic notions in Sections~\ref{S2} and \ref{S5}.

Let $A$ be any set and let $X_\alpha$, $\alpha \in A$, be any topological spaces. We prove that the direct product of groups 
$G = \prod_{\alpha\in A}G_\alpha$ is topologically transitive on the Tychonoff product 
of topological spaces $X = \prod_{\alpha \in A}X_\alpha$ if and only if every homeomorphism group $G_\alpha$, $\alpha\in A,$ is topologically 
transitive on the respective factor $X_\alpha$ (Theorem~\ref{Trans}). The analogous statement is proved for the existence of dense orbits
(Theorem~\ref{Trans1}). We get also an analog of the Birkhoff theorem (Proposition~\ref{Birkhoff}).

We investigate density of closed orbits in Section~\ref{Sec4} and show that the direct product of groups 
$G = \prod_{\alpha \in A}G_\alpha$ has a dense subset of closed orbits in $X = \prod_{\alpha \in A}X_\alpha$ if and only if 
for every $\alpha \in A$, the group $G_\alpha$ has a dense subset of closed orbits in $X_\alpha$ (Theorem~\ref{Closed}).

In Section \ref{S5} we recall the definition of the product of a countable family of metric spaces.

Section \ref{S6} is devoted to the sensitive dependence of group actions on initial conditions. Recall that a topological space 
$X$ is a Baire space if every countable intersection of open dense subsets of $X$ is dense in $X$ \cite[Def. 8.2]{Kech}. A topological 
space $X$ is referred to a completely metrizable space, if it admits an agreed complete metric \cite[Def.~3.1]{Kech}. According to the 
Baire category theorem, every completely metrizable space is a Baire space. Recall that a separable space homeomorphic to a complete 
metric space is referred to as a Polish space. Consequently Polish spaces and, in particular, compact metric spaces are Baire spaces. A 
Hausdorff topological space is called locally compact if every its point has an open neighborhood with the compact closure. Emphasize 
that compact topological spaces as well as topological manifolds are locally compact. We prove the following theorem on sufficient 
conditions for the sensitivity of homeomorphism groups.

\begin{theorem}\label{Sensation2} Let $(X, d)$ be a locally compact metric Baire space. If a homeomorphism group $G$ of $X$ 
satisfies the following two conditions: 
\begin{enumerate}
\item[(1)] there exists a dense non-closed orbit of the group $G$ in $X$ ({\it the existence of a dense orbit});
\item[(2)] the union of minimal sets of $G$ is a proper dense subset of $X$ ({\it the density of minimal sets}),
\end{enumerate}
then $G$ is sensitive to initial conditions. Moreover, $G$ is sensitive in every metric space $(X, \rho)$ such that $\rho$ and $d$ 
define the same metric topology on $X$.
\end{theorem}

For continuous actions of topological $C$-semigroups $S$ on a metric space $(X, d)$ under the additional assumption
of compactness of minimal sets whose union is everywhere dense in $X$, sensitivity of $S$ was proved by Kontorovich and
Megrelishvili~\cite{KM}.

Recall that an $n$-dimensional topological manifold is a Hausdorff topological space with a countable base, locally homeomorphic to $\mathbb R^n.$
As topological manifolds are locally compact Polish spaces, the results of Theorem~\ref{Sensation2} are applicable to them. 
 
The following important statement is a direct corollary from Theorem~\ref{Sensation2}, which is represented as a theorem because of 
the importance. Note that it is a generalization of the main result of \cite{BBC} in the case of invertible dynamical systems. 

\begin{theorem}\label{Sens1} Let $(X, d)$ be a locally compact metric Baire space. 
If a homeomorphism group $G$ acts chaotically on $X$ in sense of Definition~\ref{Chaos2}, then the group $G$ is sensitive to initial conditions.
\end{theorem}

Let $G_i$, $i \in J \subset\mathbb N$, be a homeomorphism group of a metric space $X_i$, and on the Tychonoff 
product $X = \prod_{i\in J} X_i$ the canonical action of the direct product of groups $G = \prod_{i\in J} G_i$ is given. We prove that, 
in contrast to the transitivity and density of closed orbits, in order for the canonical action of the group $G$ on the product $X$ 
to be sensitive to initial conditions, it is sufficient to have one group $G_n$, $n\in J,$ which is sensitive to initial conditions on $X_n$ 
(Theorem~\ref{Prod_sens}). In the case when the index set $J$ is finite, this condition is also necessary (Theorem~\ref{Prod_sens2}).

In Section \ref{S7} we prove the following theorem. 

\begin{theorem}\label{ThC1} For every set $A$ of indexes, let $G_\alpha$, $\alpha\in A,$ be a homeomorphism group of a topological space 
$X_\alpha$, and on the Tychonoff product $X = \prod_{\alpha\in A} X_\alpha$ the canonical action of the product of groups 
$G = \prod_{\alpha\in A} G_\alpha$ is given. Then the group $G$ acts chaotically on $X$ 
if and only if every group $G_\alpha$, $\alpha \in A$, acts chaotically on $X_\alpha$.
\end{theorem}
For a countable index set $A$ we prove Theorems~\ref{T1} and \ref{T2} on relationship between sensitivity of groups $G_\alpha,$ 
$\alpha\in A$, and $G$.

The application to compact metrizable spaces are considered (Section \ref{S7_2}).
In particular, we prove the following theorem.

\begin{theorem}\label{TC2} Let $G_i$, $i\in \mathbb N$, be a countable group of homeomorphisms of a metrizable compact
space $X_i$. Assume that every $G_i$ acts chaotically on $X_i$. Then:
\begin{enumerate}
	\item[(1)] the canonical action of the product of groups $G = \prod_{i\in \mathbb N} G_i$ is chaotic on the Tychonoff product 
  $X = \prod_{i\in \mathbb N} X_i$;
  \item[(2)] exists a dense subset $F \subset X$ which is the union of continual compact orbits, and every such orbit is a perfect subset of $X$;
  \item[(3)] exists a dense continuum orbit of the group $G$ in $X$;
	\item[(4)] all groups $G_i$, $i\in \mathbb N$, and $G$ are  residually finite; 
  \item[(5)] all groups $G_i$, $i\in \mathbb N$, and $G$ are sensitive to initial conditions (respectively every metric metrizing $X_i$);
	\item[(6)] if each group $G_i$ has a fixed point, then the union of the finite orbits of group $G$ is dense in $X$, and $G$ has a fixed point.
\end{enumerate}
\end{theorem}

Since every topological manifold is a locally compact Polish space, then all results of our work for locally compact Polish spaces are 
applicable to topological manifolds.

The Sections \ref{S8}--\ref{S9} contain the construction of families of homeomorphism groups of various topological spaces.

In Section \ref{S8} we check chaoticity of the group generated by the full $N$-shift of the space of bi-infinite sequences $\Sigma^N$ of $N$ symbols. 
This allows us to get series of new chaotic groups of homeomorphisms of different finite and infinite products of spaces $\Sigma^{N_i}$, $N_i\in\mathbb N$. 
Emphasize, that the space $\Sigma^N$ is homeomorphic to the $(2N - 1)$-ary Cantor set.

In Section \ref{S9} we construct numerous examples of chaotic homeomorphism groups of topological manifolds including noncompact manifolds. 
Using the method from \cite{CarDEKP}, we construct a countable series of examples of chaotic groups of homeomorphisms, isomorphic to the group 
$\mathbb Z$, on every closed surface as well as on various noncompact surfaces, examples of which are shown in Figures~\ref{Loch_ness_monster} 
and \ref{Cantor_tree}.
\begin{figure}[H]
\centerline{\includegraphics[width=0.6\columnwidth]{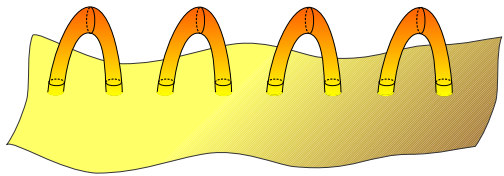}}
\caption{The Loch Ness monster.}
\label{Loch_ness_monster}
\end{figure}
Emphasize that all examples of chaotic group of homeomorphisms on noncompact topological manifolds are new and they are represented for the first time.
\begin{figure}[H]
\centerline{\includegraphics[width=0.6\columnwidth]{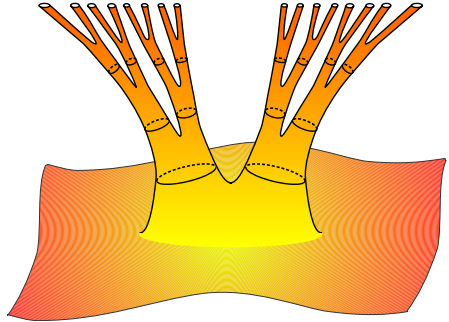}}
\caption{The surface homeomorphic to the plane without the Cantor set.}
\label{Cantor_tree}
\end{figure}
We use the obtained chaotic actions as building blocks for constructions of chaotic actions of homeomorphism groups on Tychonoff 
products of topological manifolds. Due to the result of G. Cairns and 
A. Kolganova \cite{CarK} and Theorem~\ref{ThC1}, every triangulable closed manifold of arbitrary dimension can be taken as a factor on which 
arbitrary countably generated free groups acts chaotically. In particular, we get a continuum set of examples of chaotic actions of homeomorphism 
groups on infinite dimensional topological manifolds.

\vspace{3mm}
{\bf Notations.}
If a group $G$ acts on a set $X$, we denote by $g.x$ the action of an element $g \in G$ on a point $x \in X$. By $G.x$ we denote the orbit of $x$ 
with respect to $G$.

We use the notations $D_r(x) = \{y \in X|\,\, d(x, y) < r\}$ for the open ball of a radius $r$ with the center at $x$ and 
$D_\varepsilon(B) = \{y \in X|\,\, d(y, B) < \varepsilon\} = \bigcup_{b \in B} D_\varepsilon(b)$ for the $\varepsilon$-neighborhood 
of a subset $B$ in a metric space $(X, d)$.

\vspace{3mm}
{\bf Assumptions.} Inclusions do not exclude equality. All neighborhoods are assumed to be open. By a countable set we mean an infinite 
countable set as well as a finite set.

\section{The canonical action of the direct product of groups}\label{S2}
\subsection{The Tychonoff  product of topological spaces}
Let $A$ be an arbitrary set, let $\{X_\alpha\,|\,\alpha\in A\}$ be a family of any sets. The direct (Cartesian) product $X = \prod_{\alpha\in A}X_\alpha$ 
is the set of all maps $x: A\to\bigcup_{\alpha\in A}X_\alpha$ such that $x(\alpha) \in X_\alpha$ for any $\alpha\in A$. If $x\in X$, then the point 
$x(\alpha)\in X_\alpha$ is denoted by the symbol $x_\alpha$ and is called the $\alpha$-coordinate of the element $x$. The symbol $\{x_\alpha\}$ will denote 
the point of the product $X$, $\alpha$-coordinate of which is the point $x_\alpha \in X_\alpha$.

Let $\{(X_\alpha, \tau_\alpha)\,|\,\alpha\in A\}$ be a family of topological spaces. Assume that $X = \prod_{\alpha\in A} X_\alpha$ is provided by the weakest
topology $\tau$ such that all projections $\pi_\alpha: \prod_{\alpha\in A} X_\alpha\to X_\alpha$, $\pi_\alpha(\{x_\alpha\}) := x_\alpha$, are continuous. This topology $\tau$ is 
called by the {\it Tychonoff } topology. Note that $$\zeta = \{\pi_\beta^{-1}(U)\subset X\,|\, U\in\tau_\beta, \beta\in A\}$$ is a subbase of $\tau$. 
The base of the Tychonoff topology formed by all possible finite intersections of subsets of $\zeta$, is called the {\it canonical} base.

The topological space $(X, \tau)$ is called the {\it Tychonoff  product of topological spaces}
$(X_\alpha, \tau_\alpha)$ and is denoted by $(X, \tau) = \prod_{\alpha\in A}(X_\alpha, \tau_\alpha)$.

\subsection{A direct product of groups and its canonical action}
Let $G_\alpha$, $\alpha\in A,$ be a family of groups. On the direct product of sets $G = \prod_{\alpha\in A} G_\alpha$, the group operation is 
introduced as follows
$$\Psi: G\times G\to G, \,\,\,\,\Psi(\{g_\alpha\}, \{h_\alpha\}) = \{g_\alpha\cdot h_\alpha\}\,\,\,\,  
\forall(\{g_\alpha\}, \{h_\alpha\})\in G\times G,$$
where $g_\alpha\cdot h_\alpha$ is the product of elements $g_\alpha$ and $h_\alpha$ in $G_\alpha$. The group $G = \prod_{\alpha \in A} G_\alpha$ 
is referred to as a direct product of groups $G_\alpha$, $\alpha\in A$.

Let $X = \prod_{\alpha\in A} X_\alpha$ be a direct product of sets. Assume that for every $\alpha\in A$ a group $G_\alpha$ acts on $X_\alpha$.
Consider the direct product of groups  $G =\prod_{\alpha\in A} G_\alpha$. Then the following action of the group $G$ on $X$ is defined
$$\Phi: G\times X\to G, \,\,\,\,\Phi(\{g_\alpha\}, \{x_\alpha\}) = \{g_\alpha.x_\alpha\}\,\,\,\, \forall(\{g_\alpha\}, \{x_\alpha\})\in G\times X.$$
We call this action the {\it canonical action of the direct product of groups $G =\prod_{\alpha\in A} G_\alpha$ on the direct product of sets
 $X = \prod_{\alpha\in A} X_\alpha$}.

\vspace{3mm}
Further in this work we assume that the direct product of groups $G =\prod_{\alpha \in A} G_\alpha$ acts on $X = \prod_{\alpha\in A} X_\alpha$
canonically.

\section{Transitivity of the canonical product of homeomorphism groups}\label{S3}
\begin{definition}\label{Tr} A homeomorphism group $G$ of a topological space $X$ is called topologically transitive on $X$
if for every nonempty open subsets $U$ and $V$ in $X$ there exists such an element $g\in G$ that $$g(U)\cap V \neq\varnothing.$$
\end{definition}

\begin{theorem}\label{Trans} Let $A$ be an arbitrary set of indexes. For every $\alpha\in A$, a homeomorphism group $G_\alpha$ of the 
topological space $X_\alpha$ is topologically transitive on $X_\alpha$ if and only if the direct product of groups $G = \prod_{\alpha\in A}G_\alpha$
topologically transitive on the Tychonoff product of topological spaces $X = \prod_{\alpha \in A}X_\alpha$.
\end{theorem}
\begin{proof} Suppose that for every $\alpha\in A$, the group of homeomorphisms $G_\alpha$ of the topological space $X_\alpha$ is topologically transitive
on $X_\alpha$. Let $X = \prod_{\alpha \in A}X_\alpha$ be the Tychonoff product of topological spaces $X_\alpha$. Show that for every nonempty open
subsets $U$ and $V$ in $X$ there exists an element $g = \{g_\alpha\} \in G$ such that $g(U) \cap V \neq \varnothing$. It is enough to prove this 
fact for every sets $U$ and $V$ from the canonical base of the Tychonoff topology on $X$. Let $U = \prod_{\alpha \in A}U_\alpha$, 
$V = \prod_{\alpha \in A}V_\alpha$ where $U_\alpha$ and $V_\alpha$ are nonempty open subsets in $X_\alpha$, and for some finite subsets
$A_1, A_2\subset A$ the equalities are fulfilled
$$U_\alpha = X_\alpha \,\,\, \forall \alpha \in A \setminus A_1, \,\,\,\,\,\, V_\beta= X_\beta\,\,\, \forall \beta \in A\setminus A_2.$$
The topological transitivity of $G_\alpha$ on $X_\alpha$ implies the existence of elements $g_\alpha \in G_\alpha$ satisfying
$g_\alpha (U_\alpha) \cap V_\alpha \neq \varnothing$ for all $\alpha \in A$. Put $g = \{g_\alpha\}\in G,$ then
$g(U) \cap V = \prod_{\alpha \in A}(g_\alpha(U_\alpha) \cap V_\alpha)\neq\varnothing.$ 
Thus, the group $G$ is topologically transitive on  $X$.

The opposite. Suppose that the homeomorphism group $G =\prod_{\alpha\in A}G_\alpha$ is topologically transitive on the Tychonoff product 
$X = \prod_{\alpha \in A}X_\alpha$. Fix an arbitrary element $\delta\in A.$ Let $U_\delta$ and $V_\delta$ be any nonempty open subsets in 
$X_\delta$. Let $U := \pi^{-1}(U_\delta)$ and $V := \pi^{-1}(V_\delta)$, hence $U$ and $V$ are nonempty open subsets in $X$. Since the group 
$G$ acts topologically transitive on $X$, then there is an element $g = \{g_\alpha\} \in G$ such that 
$\varnothing \neq g(U) \cap V = \prod_{\alpha \in A}(g_\alpha(U_\alpha) \cap V_\alpha)$, therefore
$g_\delta(U_\delta) \cap V_\delta \neq \varnothing.$ This implies topological transitivity of the group $G_\delta$ on $X_\delta$ for every 
$\delta \in A$.
\end{proof}

\begin{theorem}\label{Trans1} Let $A$ be an arbitrary index set and let $X = \prod_{\alpha \in A}X_\alpha$ be the Tychonoff product of 
topological spaces $X_\alpha$. Assume that $G_\alpha$ is a homeomorphism group of $X_\alpha$. Then the direct product of groups 
$G = \prod_{\alpha \in A}G_\alpha$ has a dense orbit in $X$ if and only if the group $G_\alpha$ has a dense orbit in $X_\alpha$
for every $\alpha \in A.$
\end{theorem}
\begin{proof} As known \cite[Prop. 2.3.3]{Eng}, for every family of subsets $B_\alpha\subset X_\alpha$ in the product
$X = \prod_{\alpha \in A}X_\alpha$ the closures satisfy the following relation
\begin{equation}\label{Cl}
\overline{\prod_{\alpha \in A}B_\alpha} = \prod_{\alpha \in A}\overline{B_\alpha}.
\end{equation}
Since the orbit $G.x$ of $x = \{x_\alpha\}\in X$ is equal to the product of orbits $G_\alpha.x_\alpha$, i.e.,  
$G.x = \prod_{\alpha \in A}G_\alpha.x_\alpha$, then taking into account \eqref{Cl} we get a chain of equalities
$\overline{G.x} = \overline{\prod_{\alpha \in A}G_\alpha.x_\alpha} = \prod_{\alpha \in A}\overline{G_\alpha.x_\alpha}$. Therefore
\begin{equation}\label{CL2}
\overline{G.x} = \prod_{\alpha \in A}\overline{G_\alpha.x_\alpha}\,\,\, \,\,\,\,\,\,\forall x = \{x_\alpha\}\in X = \prod_{\alpha \in A}X_\alpha.
\end{equation}
Using \eqref{CL2} it is easy to obtain a statement of the theorem being proved.
\end{proof}

{\bf An analog of the Birkhoff transitivity theorem}

\vspace{2mm}
If the action of the group $G$ on the space $X$ has a dense orbit, then it is topologically transitive. Indeed, let
$U$ and $V$ are any open nonempty subsets of $X$ and let $\overline{G.x} = X$ for some $x\in X$. Then there are elements
$g_1, g_2\in G$ such that $g_1.x\in U$ and $g_2.x\in V$. It follows that $g.U \cap V \neq \varnothing$ where $g = g_2 g_1^{-1}\in G$.

According to \cite[Prop. 1]{CairnsKolgN}, if the group $G$ is topologically transitive on a Baire space $X$ with a countable base, 
then there exists a point $x\in X$ with a dense orbit. Therefore we get the following analog 
of the Birkhoff theorem for homeomorphism groups of Baire spaces with a countable base. 

\begin{proposition}\label{Birkhoff} If $G$ is a homeomorphism group of a topological space $X$, then the existence of a dense orbit
of $G$ implies topological transitivity of $G$.

When $X$ is a Baire space with a countable base, the converse is also true.
\end{proposition}

In particular, for Polish spaces $X$, the existence of a dense orbit of $G$ on $X$ is equivalent to topological transitivity of $G$.

\begin{example}
Let $\{e_k\}_{k = \overline{1, n}}$, be a basis of the vector space $\mathbb R^n$ for $n\geq 1$. Define homeomorphisms 
of the Euclidean space $\mathbb R^n$ by the following equalities: $g_k(x) = x + e_k$, for $k = \overline{1, n}$, and $g_{n+1}(x) = \lambda x$ where $\lambda > 1$ 
for all $x\in\mathbb R^n$. Consider the homeomorphism group $G = \langle g_k\,|\, k = \overline{1, n+1} \rangle$. According to \cite[Prop. 16]{ZhM}, the 
homeomorphism group $G$ is topologically transitive, and every its orbit is dense in $\mathbb R^n$. In other words, $\mathbb R^n$ is a 
minimal set of the group $G.$ Note that there is no transitive subgroup of $G$ with the number of generators less than $n+1$.
\end{example}

\section{Density of closed orbits}\label{Sec4}

\begin{theorem}\label{Closed} Let $A$ be an arbitrary index set and let $X = \prod_{\alpha \in A}X_\alpha$ be the Tychonoff product of topological 
spaces $X_\alpha$. Assume that $G_\alpha$ is a homeomorphism group of $X_\alpha$. Then the direct product of groups 
$G = \prod_{\alpha \in A}G_\alpha$ has a dense union of closed orbits in $X$ if and only if for every $\alpha \in A$, 
the group $G_\alpha$ has a dense union of closed orbits in the topological space $X_\alpha$.
\end{theorem}
\begin{proof} Let $x = \{x_\alpha\}\in X$, then $G.x = \prod_{\alpha \in A}G_\alpha.x_\alpha$. According to
\eqref{CL2}, we get
\begin{equation}\label{Closure}
\overline{G.x} = G.x\,\,\Longleftrightarrow\,\, \overline{G_\alpha.x_\alpha} = G_\alpha.x_\alpha\,\,\, \forall \alpha\in A.
\end{equation}
This means that an orbit $G.x$ is closed in $X$ if and only if the orbit $G_\alpha.x_\alpha$  is closed in $X_\alpha$ 
for every $\alpha\in A.$ Let $B$ be the union of all closed orbits of $G$ in $X$. Denote by $B_\alpha$ the union of all closed orbits of 
$G_\alpha$ in $X_\alpha.$ Therefore $B = \prod_{\alpha \in A}B_\alpha$. Suppose that $B$ is dense in $X$, hence applying the equality 
\eqref{Cl}, we get the following chain of equalities
 $X = \overline{B} = \overline{\prod_{\alpha \in A}B_\alpha} = \prod_{\alpha \in A}\overline{B_\alpha}.$ Consequently 
$X_\alpha = \overline{B_\alpha}$, i.e. $B_\alpha$ is dense in $X_\alpha$.

Conversely, let for every $\alpha\in A$ the subset of $B_\alpha$ be dense in $X_\alpha$. As $B = \prod_{\alpha \in A}B_\alpha$, applying the 
equality \eqref{Cl}, we have $\overline{B} = \prod_{\alpha \in A}\overline{B_\alpha} = \prod_{\alpha \in A}X_{\alpha} = X.$
This means that the union of all closed orbits of $G$ is dense in $X$.
\end{proof}
 
\section{Countable products of metric spaces}\label{S5}

\subsection{The direct product of a countable family of metric spaces}
{\bf The direct product of two metric spaces.}
Let $(X_1, d_1)$ and $(X_2, d_2)$ be metric spaces. A metric $d$ on the product of two metric spaces 
$X_1 \times X_2 = \{(x_1, x_2)\,|\, {x_1 \in X_1,} {x_2 \in X_2}\}$ may be introduced in the following ways (\cite[Sec. 4.2]{Deza}):
\begin{enumerate}
	\item $d(x, y) = \sqrt[p]{d_1^p(x_1, y_1) + d_2^p(x_2, y_2)}$ where $p \geq 1$.
  \item $d(x, y) = \max\{d_1(x_1, y_1), d_2(x_2, y_2)\}$.
  \item $d(x, y) = \mathop \sum \limits_{i = 1}^2 \frac{1}{2^i} \frac{d_i(x_i, y_i)}{1 + d_i(x_i, y_i)}$.
\end{enumerate}

All these methods can be easily extended to the case of the product of any finite number of factors.

{\bf The direct product of a countable family of metric spaces.}
Based on any metric $d$ on the set $X$, we can get a metric bounded by the number $1,$ by the formula
\begin{equation} \label{restrict_metrics}
\widetilde{d}(x, y) = \frac{d(x, y)}{1 + d(x, y)}\,\,\,\, \forall x, y\in X,
\end{equation} 
and metric topologies defined by the metrics $d$ and $\widetilde{d}$ are coincided.

Let $(X_i, d_i)$, $i\in\mathbb{N}$, be a countable family of metric spaces. Metric on the product
$X = \prod_{i\in\mathbb{N}} X_i$ can be given by the equality (\cite[Th. 4.2.2]{Eng}):
\begin{equation} \label{coutable_prod_metrics}
d(x, y) = \mathop \sum \limits_{i = 1}^\infty \frac{1}{2^i} \widetilde{d}_i(x_i, y_i),
\end{equation}
where $\widetilde{d}_i$ is defined by the formula \eqref{restrict_metrics} if the metric $d_i$ is not bounded by a number $1$,
otherwise $\widetilde{d}_i = d_i$.

Emphasize that the topology on $X$ generated by the metric $d$, defined by \eqref{coutable_prod_metrics}, coincides with the Tychonoff topology of the product
$X = \prod_{i\in\mathbb{N}} X_i$ of topological spaces $X_i$.

\begin{definition} The metric $d$ on $X = \prod_{i\in\mathbb{N}} X_i$ given by the formula \eqref{coutable_prod_metrics},
is called by the {\it direct product} of metrics $\widetilde d_i$. The metric
space $(X, d)$ is called the {\it direct product} of countable family of metric spaces $(X_i, \widetilde d_i)$ and it is
denoted by $(X, d) = \prod_{i\in\mathbb{N}} (X_i, \widetilde d_i)$.
\end{definition}

\begin{remark} As it is known \cite{Deza}, the metric \eqref{coutable_prod_metrics} is a special case of the metric
$d(x,y) = \mathop \sum \limits_{i = 1}^\infty A_i \widetilde{d}_i(x_i, y_i)$, where the series $\mathop \sum \limits_{i = 1}^\infty A_i$
converges and all of its members are positive.
\end{remark}

\subsection{Nonmetrizability of the product of an uncountable family of topological spaces}
Let $(M,d)$ be a metric space and $D_{r}(x) = \{z\in M\,|\,d(x,z) < r\}$ be the ball of radius $r > 0$ centered at $x$.
Recall that a topological space satisfies the first axiom of countability if it has a countable base of topology at each point.
Every metric space $(M,d)$ has a countable base $\Sigma_x = \{D_{1/n}(x)\,|\,n\in\mathbb N\}$ at each point $x\in M.$
According to \cite[Cor. 4.2.4]{Eng}, an uncountable product of metrizable spaces such that every of them contains at least two points, 
does not satisfy the first axiom of countability. Thus an uncountable product of such spaces is not metrizable.

Therefore, we will consider further only products of a countable family of metric spaces in investigations of sensitivity of homeomorphism groups.

\section{The sensitivity of homeomorphism groups to initial conditions}\label{S6}

\subsection{Properties of sensitivity}
\begin{definition}\label{Chu1} A homeomorphism group $G$ of a metric space $(X, d)$ is called {\it sensitive to initial
conditions at a point} (or, for short, {\it sensitive at a point}) $x\in X,$ if there exists a number $\delta = \delta(x) > 0$ such 
that for every neighborhood $U_x$ of $x$ there exists an element $g\in G$ satisfying the following inequality:
\begin{equation}
\diam(g.U_x) \geq \delta.
\end{equation}
A group $G$ is called {\it pointwise sensitive to initial conditions} (or, for short, {\it pointwise sensitive}), if it is sensitive at every point $x\in X.$
\end{definition}

\begin{definition}\label{Chu2} A homeomorphism group $G$ of a metric space $(X, d)$ is called {\it sensitive to initial
conditions} (or, for short, {\it sensitive})  if there exists a number $\delta > 0$ such that for each open subset $U\subset X$ 
there exists an element $g\in G$ satisfying the following inequality:
\begin{equation}
\diam(g.U) \geq \delta.
\end{equation}
The number $\delta$  is referred to as {\it the sensitivity constant} for $G$.
\end{definition}

Emphasize that every group $G$ satisfying Definition~\ref{Chu2}, satisfies also Definition~\ref{Chu1} of the pointwise sensitivity. 
Note that without violating generality, in Definitions~\ref{Chu1} and \ref{Chu2} we may consider only neighborhoods from the base of the metric topology, 
that is, neighborhoods of the form $D_\varepsilon(x)$, $x\in X$, $\varepsilon > 0.$ 

\begin{example}\label{E1} Recall that a homeomorphism group $G$ of a metric space $(X, d)$ is said to be {\it expansive} on $X$, 
if there exists a constant $c > 0$ such that for every $x \neq y$ in $X,$ there is $g\in G$ satisfying $d(g.x, g.y) > c$.
Such a constant $c$ is called an {\it expansivity constant} of this group \cite{Bar}.
Every expansive group $G$ of homeomorphisms of a metric space $(X, d)$ is sensitive to initial conditions, and the role of $\delta$
in Definition~\ref{Chu2} plays the expansivity constant $c.$
\end{example}

\begin{proposition}\label{PP} Let $G$ be a homeomorphism group of a metric Baire space $X$ having a dense non-closed orbit. 
Then $G$ is pointwise sensitive if and only if it has sensitive dependence on initial conditions.
\end{proposition}
\begin{proof} As the sensitivity implies pointwise sensitivity, prove the inverse. Let $X$ be a metric Baire space. Assume that a group $G$ 
is pointwise sensitive and has a dense non-closed orbit. For each $n\in\mathbb N$ consider the following subset
\begin{equation}\label{E8}
V_n =\{x\in X\,|\, \exists\, \varepsilon > 0 : \diam(g.D_{\varepsilon}(x)) < 1/n\,\,\, \,\forall g\in G\}.
\end{equation}
Note that $V_n$ is an open $G$-invariant subset in $X$, and $V_n \supset V_{n+1} \supset V_{n+2} \supset ...$ Since every nonempty
open set contains a point with dense orbit, then if $V_n \neq \varnothing$ for some $n$, due to $G$-invariance of $V_n$, the set $V_n$ is dense
in $X$.  Since $X$ is a Baire space, if $V_n \neq\varnothing\,\,\,\forall n\in\mathbb N$, then $\bigcap_{n\in\mathbb N}V_n$ is dense in $X$.
Emphasize that each point $x\in\bigcap_{n\in\mathbb N}V_n$ is not a sensitive point respectively $G$, that contradicts the assumption. Consequently,
there exists $m\in\mathbb N$ for which $V_m = \varnothing$. Therefore, $V_n = \varnothing$ for every $n\geq m.$ According to \eqref{E8}, this 
means that there exists $\delta = 1/m$ such that for every $x\in X$ and for every $\varepsilon > 0$ there exists an element $g\in G$ satisfying 
$\diam(g.D_\varepsilon(x)) \geq \delta$. By Definition~\ref{Chu2}, the group $G$ is sensitive to initial conditions on $X.$ This completes 
the proof.
\end{proof}

\begin{remark}\label{RChu} For a continuous action of a topological group $G$ on a compact metric space $(X, d)$, as indicated by F. Polo
\cite[Prop.~1.3]{Polo} one can prove Proposition~\ref{PP} using ideas from \cite{AAB}. In fact, we have implemented such a possibility 
under a more general assumption, replacing the compactness condition of the metric space $X$ with the assumption that $X$ is a Baire space.
\end{remark}

\begin{lemma}\label{LS1} If $x$ is a sensitive point of a homeomorphism group $G$ of a metric space $(X, d)$ with a sensitive constant 
$\delta = \delta(x)$, then all points of the closure $\overline{G.x}$ of the orbit $G.x$ are also sensitive with the same sensitive 
constant $\delta.$
\end{lemma} 
\begin{proof} As $x$ is a sensitive point, there exists $\delta = \delta(x) > 0$ such that for every neighborhood $U = U_x$ there 
are points $y, z\in U$ and $g\in G$ satisfying $d(g.y, g.z) > \delta(x).$ Pick $x'\in G.x.$ Let $U' = U'_{x'}$ be an arbitrary 
neighborhood at $x'.$ Let $x' = g'.x, \, g'\in G.$ The sensitivity of $x$ implies that in the neighborhood $g'^{-1}(U')$ of $x$ there exist 
points $g'^{-1}(y')$, $g'^{-1}(z')$, $y', z'\in U'$ and there is $\widehat{g}\in G$, satisfying the inequality $(h.y', h.z') > \delta(x)$ for 
$h = \widehat{g}g'^{-1}\in G.$ Consequently $\delta = \delta(x)$ is a sensitive constant for every point of the orbit $G.x.$

For each point $v\in\overline{G.x}$ and for every its neighborhood $V = V(v)$ there is $g\in G$ for which $g.x\in V.$ As $g.x\in G.x$, 
according to the fact proved above, there are points $y, z\in V$ and an element $g'\in G$ such that $d(g'.y, g'.z) > \delta(x).$ This 
means that $\delta(x)$ is a sensitive constant at $v.$ Thus, $\delta = \delta(x)$ is a common sensitive constant for every points from the 
closure $\overline{G.x}$.
\end{proof}

\begin{lemma}\label{LS2} Let $(X, d)$ be a metric space. Assume that a group $G$ of homeomorphisms of $X$ has a dense non-closed
orbit. Then $G$ is sensitive to initial conditions if and only if there exists a sensitive point with dense orbit.
\end{lemma}
\begin{proof} According the condition, there exists a point $x$ with a dense orbit $G.x.$ Assume that $G$ is sensitive. Therefore every 
its point is sensitive, hence $x$ is a sensitive point with the dense orbit $G.x.$

Converse, let there exists a sensitive point with dense orbit $G.x$, i.e. $X = \overline{G.x}.$ According to Lemma~\ref{LS1}, 
$\delta = \delta(x)$ is a sensitive constant for $G$.
\end{proof} 

Let us use the terminology, generally accepted in the theory of dynamical systems, for groups of homeomorphisms.

\begin{definition} Let $G$ be a homeomorphism group of a metric space $(X, d)$. An action of $G$ is called {\it minimal}, if
every orbit of $G$ is dense in $X,$ i.e. if $X$ is a minimal set of $G$.

For the group $G$, the term an {\it equicontinuous} point is the synonym of an {\it insensitive} point.  An action of $G$ is 
called {\it equicontinuous} (or $G$ is equicontinuous), if every point $x\in X$ is equicontinuous. An action of $G$ is 
called {\it almost equicontinuous} (or $G$ is almost equicontinuous), if the set of equicontinuous points of $G$ is dense in $X$.
\end{definition}
Introduce the following notations. Let $\cal NS$ be the set of all insensitive points and let $\cal D$ be the set of all points with
dense orbits of $G$ in $X$.  

\begin{theorem}\label{TS1} Let $(X, d)$ be a metric space and let $G$ be a homeomorphism group of $X$. Assume that 
$G$ has a dense non-closed orbit, and $G$ is insensitive. Then 
\begin{equation}\label{EE}
{\cal NS = \cal D}.
\end{equation}

If, moreover, $(X, d)$ is a metric Baire space, then $\cal D$ is a dense $G_\delta$-set, coinciding with $X$ in the case when $G$ is minimal.
\end{theorem}
\begin{proof} Assume that $G$ is insensitive to initial conditions. As $G$ has a dense non-closed orbit, ${\cal D}\neq\varnothing$, hence, by 
Lemma~\ref{LS2}, the inclusion ${\cal D}\subset\cal NS$ holds.

If $G$ is minimal, then the inclusion ${\cal D}\subset\cal NS$ implies ${\cal D} = X = \cal NS$, hence the equality \eqref{EE} is true.

Let $G$ is non-minimal. Assume that the equality \eqref{EE} is not true. Consequently there is an insensitive point 
$x\in X$ such that its orbit $G.x$ is not dense in $X$. Then $U = X\setminus{\overline{G.x}}$ 
is nonempty open subset in $X$, hence there exists $a\in U$, and $d(a, \overline{G.x}) > 0.$ Put $\delta = \frac{1}{2}d(a, \overline{G.x}).$ 
According the condition, $G$ has a dense non-closed orbit. Therefore for every neighborhood $V_x$ of $x$ there exists a point $y\in V_x$ having a 
dense orbit. Hence there exist $y'\in G.y \cap D_\delta(a)$ and an element $g\in G$ for which $y' = g.y$. Show that $\diam(g.V_x) \geq \delta.$ 
Suppose the opposite, i.e. $\diam(g.V_x) < \delta.$ Using the triangle inequality in the metric space $(X,d)$, we get
$$d(a, g.x) \leq d(a, y') + d(g.y, g.x) < \delta + \delta = 2\delta = d(a, \overline{G.x})$$
which contradicts the definition of distance $d(a, \overline{G.x}).$ The contradiction proves inequality 
$\diam(g.V_x) \geq \delta.$ Therefore, $x$ is sensitive to initial conditions, that contradicts our assumption. 
Hence we have shown the inclusion $\cal NS \subset\cal D$. Thus we proved that ${\cal D} = \cal NS.$ 

Now assume that $(X, d)$ is a metric Baire space and show that $\cal D$ is a $G_\delta$-subset in $X$. As above, by $U_x$ we denote 
a neighborhood of $x$. Consider 
\begin{equation}\label{EE1}
V_n =\{x\in X\,|\, \exists\, U_x : \diam(g.U_x)\leq 1/n\,\,\, \,\forall g\in G\}.
\end{equation}
Note that $V_n$ is an open $G$-invariant subset in $X$ for each $n\in\mathbb N$, and $V_n \supset V_{n+1} \supset V_{n+2} \supset ...$ 
According to Definition~\ref{Chu1}, ${\cal NS}\subset V_n$ for every $n\in\mathbb N$, hence ${\cal NS}\subset\bigcap_{n\in\mathbb N}V_n.$
The inclusion $\bigcap_{n\in\mathbb N}V_n\subset{\cal NS}$ is also true. Thus, ${\cal NS} = \bigcap_{n\in\mathbb N}V_n.$
Since every open set contains a point with dense orbit, and $V_n \neq \varnothing$ for every $n$, due to $G$-invariance of $V_n$, the set $V_n$ 
is dense in $X$. As $X$ is a Baire space, the intersection $\bigcap_{n\in\mathbb N}V_n = {\cal NS = \cal D}$ is a dense $G_\delta$-subset in $X$. 
This completes the proof of the theorem.
\end{proof}

We have the following two direct corollaries from Theorem~\ref{TS1}.
\begin{corollary}\label{CS1} Let $(X, d)$ be a metric Baire space. Let $G$ be a group of homeomorphisms of $X$ having
a dense non-closed orbit. Then 
$${\cal NS}\neq\varnothing \,\,\,\Leftrightarrow\,\,\, \overline{\cal NS} = X.$$
\end{corollary}

\begin{corollary}\label{CS2} Let $(X, d)$ be a metric Baire space and let $G$ be a group of homeomorphisms of $X$
with a dense non-closed orbit. Then $G$ is either almost equicontinuous or sensitive.
\end{corollary}

\begin{remark}\label{RS1} Let $(X, d)$ be a metric Baire space with a countable base and let $G$ be a topologically transitive group of 
homeomorphisms of $X.$ Then, according to Proposition~\ref{Birkhoff}, $G$ has a dense orbit. Therefore for such $G$ Theorem~\ref{TS1} and 
its Corollary~\ref{CS2} are applicable.
\end{remark}

\begin{remark}\label{RS2} For topologically transitive dynamical systems $(X, T)$ where $T = \{f^n\}_{n\in\mathbb N}$,
$f: X\to X$ is a continuous map of an infinite compact metric space $(X, d)$, the statements similar to Theorem~\ref{TS1} and 
Corollary~\ref{CS2} were proved in (\cite[{Th. 2.4}]{AAB}).
\end{remark}

\subsection{Proof of Theorem~\ref{Sensation2}}
Assume that $G$ satisfies the conditions of Theorem~\ref{Sensation2}, but $G$ is not sensitive to initial conditions. Due to
Proposition~\ref{PP}, there exists an equicontinuous point of $G$ in $X$. Therefore the conditions of Theorem~\ref{TS1} are 
satisfied. Since there exists a dense orbit, $\cal D\neq\varnothing.$ Consequently, according to Theorem~\ref{TS1}, 
$\cal NS = \cal D\neq\varnothing,$ hence there exists equicontinuous point 
$x\in X$ with the dense orbit $G.x$. This means that for every $\eta > 0$ there exists a neighborhood $U = U_x$ of $x$ such that 
$\diam(g. U)<\eta$ for every $g\in G$. Denote by $M$ the union of all minimal sets of the group $G$. The set $M$ is dense in $X$ 
by condition $(2)$ of Theorem~\ref{Sensation2}, hence there exists $z\in U\cap M.$ Consequently
\begin{equation}\label{S32}
d(g.x, g.z) < \eta\,\,\,\,\,\,\,\, \forall\, g\in G,
\end{equation}
in particular, setting $g = id_X$ we see that 
\begin{equation}\label{S320}
d(x, z) < \eta.
\end{equation}
Pick an arbitrary point $v\in X\setminus G.x$. According the condition of Theorem~\ref{Sensation2}, $X$ is locally compact, hence
there exists an open neighborhood $U_v$ the closure $\overline{U_v}$ of which is compact. Then every neighborhood 
$D_{2\eta}(v)\subset U_v$ is also relatively compact.
As $\overline{G.x} = X,$ in the metric space $(X, d)$ there exists a sequence $\{g_k.x\}$, $g_k\in G$, converging to $v$ as 
$k\to\infty.$ So for every $\varepsilon$ satisfying the inequalities $0< \varepsilon < \eta/2$ there exists a number $k_0$ thus that
\begin{equation}\label{S33}
d(v, g_k.x) < \varepsilon\,\,\,\,\,\,\,\, \forall\, k > k_0.
\end{equation}
We consider $g_n.z$ and $g_m.z$ for $n\neq m$ as different members of the sequence $\{g_k.z\},$ $k\in\mathbb N$.

Since $g_k.U$ is an open neighborhood of a point $g_k.x,$ and $g_k.z\in g_k.U$, $\diam(g_k.U) <\eta$, then 
$$d(v, g_k.z)\leq d(v, g_k.x) + d(g_k.x, g_k.z) < \varepsilon +\eta\,\,\,\,\,\,\,\, \forall\, k > k_0.$$ 
Consequently, $g_k.z\in D_{\varepsilon +\eta}(v)\subset D_{2\eta}(v) \subset U_v$ for $k > k_0.$ Compactness of 
$\overline{D_{2\eta}(v)}$ implies the existence of a converging subsequence $g_{k_s}.z\to y$ as $s\to \infty.$ Therefore 
$y\in\overline{G.z}\cap\overline{D_{\varepsilon +\eta}(v)}\subset \overline{D_{2\eta}(v)}$ for every $\varepsilon$ satisfying the 
inequalities $0< \varepsilon < \eta/2.$ For simplicity, denote this subsequence also by $\{g_k.z\}$. Emphasize that if $y$ is an 
isolated point of the minimal set $\overline{G.z}$, then $\overline{G.z} = G.z$ is a closed orbit, and it is necessary $y = g_k.z,$ 
starting from some $k,$ and all the inequalities obtained further are correct.

Since $\overline{G.z}$ is a minimal set of $G$, and $y\in\overline{G.z},$ then $\overline{G.y} = \overline{G.z}$. Therefore 
$G.z\subset\overline{G.y}.$ As $U$ is a neighborhood of the point $z$, and $z\in\overline{G.y},$ there exists $y_0\in U\cap G.y.$ 
This implies the existence $h\in G$ such that $y_0 = h.y.$ Consider a homeomorphism $h: X\to X$. Let $v' := h(v) = h.v.$ Taking in 
account that $g_k.x\to v$ as $k\to\infty,$ we get $h.(g_k.x)\to v'$ as $k\to\infty.$ This means the existence $k_1$ such that 
\begin{equation}\label{S1}
d(h.(g_k.x), v') = d(h g_k.x, v') < \varepsilon \,\,\,\,\,\,\, \forall\, k > k_1.
\end{equation}
As $\diam(g. U)<\eta$ for every $g\in G$, we have $\diam(h.(g_k.U)) = \diam(hg_k.U) < \eta.$
Note that $h.(g_k.x), \,h.(g_k.z)\in h g_k.U,$ so we have
\begin{equation}\label{S3}
d(h g_k.x, h g_k.z) < \eta.
\end{equation}
Since $g_k.z\to y$ as $k\to\infty,$ we get $h(g_k.z)\to h.y = y_0$ as $k\to\infty,$ hence there exists $k_2$ 
such that 
\begin{equation}\label{S4}
d(y_0, h.(g_k.z)) = d(y_0, h g_k.z) < \varepsilon \,\,\,\,\,\,\, \forall\, k > k_2.
\end{equation}.

Consequently, taking in account that $0< \varepsilon < \eta/2$ and 
applying the relations \eqref{S1}, \eqref{S3} and \eqref{S4} for
 $k > max\{k_0, k_1, k_2\},$ we get
\begin{equation}\label{S37}
d(x, v') \leq d(x, y_0) + d(y_0, h g_k.z) +  d(h g_k.z, h g_k.x) + d(h g_k.x, v') < 2\varepsilon + 2 \eta < 3\eta. 
\end{equation}

\vspace{2mm}
Thus we have shown that for every $\eta > 0$ there exists $v'\in G.v$ satisfying the inequality $d(x, v') < 3 \eta,$ hence
$x\in \overline{G.v}.$ Since the closure of every orbit of a homeomorphism group is $G$-invariant, it is necessary that 
$\overline{G.x}\subset \overline{G.v}.$ Therefore $X = \overline{G.v}$ for every point $v\in X.$ This means that $X$ is a minimal 
set of $G$, that contradicts to the condition $(2)$ of the theorem being proved. The contradiction implies that 
the group $G$ is sensitive to initial conditions.  

\vspace{2mm}
As the conditions $(1)$ and $(2)$ of Theorem~\ref{Sensation2} are topological, then the sensitivity of $G$ does not depend on
a choice of metrics on the set $X$ defining the same topology as $d$. \,\,\,\,\, \qed

\vspace{2mm}
Since for Polish spaces the topological transitivity of a homeomorphism group $G$ is equivalent to the existence of 
a dense orbit, then the following statement is true.

\begin{corollary}\label{SensPolish} Let $X$ be a locally compact Polish space. If a homeomorphism group $G$ of $X$ 
satisfies the following two conditions: 
\begin{enumerate}
\item[(1)] the group $G$ is topological transitive on $X$;
\item[(2)] the union of minimal sets of $G$ is a proper dense subset in $X$,
\end{enumerate}
then $G$ is sensitive to initial conditions.
\end{corollary}

\subsection{Sensitivity to initial conditions of the direct product of homeomorphism groups}

Emphasize that all metric spaces are considered with the metric topology. We use notations from 
Section~\ref{S5}. In the case when $J$ is a finite subset of $\mathbb N$, the metric $d$ of the product 
$\prod_{i \in J}(X_i, \widetilde{d}_i)$ is defined by the formula analogous to \eqref{coutable_prod_metrics}, in which $i \in J$.

\begin{theorem}\label{Prod_sens} 
Let $(X, d) = \prod_{i \in J}(X_i, \widetilde{d}_i)$ where $(X_i, \widetilde{d}_i)$ is a metric space and let $J \subset \mathbb N$. Let
$G_i$, $i \in J$, be a homeomorphism group of $X_i$. If there exists $n \in J$ such that $G_n$ is sensitive to initial conditions 
on $(X_n, \widetilde d_n)$, then the direct product of groups $G = \prod_{i \in J} G_i$ is also sensitive to initial conditions on $(X, d)$.
\end{theorem}
\begin{proof}
Let $x = \{x_i\}$ be any point in $X$. Assume that a group $G_n$ is sensitive to initial conditions on $(X_n, \widetilde{d}_n)$. 
This means that there exists a number $\sigma > 0$ such that for the point $x_n \in X_n$ and for any $\eta > 0$ there is an element 
$\widetilde g \in G_n$ such that $\diam(\widetilde g.D_\eta(x_n)) \geq \sigma$. Therefore there exists a point $\widetilde y \in D_\eta(x_n)$ 
satisfying the inequality $\widetilde d_n(\widetilde g.x_n, \widetilde g.\widetilde y) \geq \frac{\sigma}{2}$.

Consider $y = \{y_i\} \in X$ where $y_n = \widetilde y$ and $y_i = x_i$ for $i \in J \setminus \{n\}$. Pick an element $g = \{g_i\} \in G$ 
where $g_i$ is an arbitrary element of $G_i$ where $i \in J \setminus \{n\}$ and $g_n = \widetilde g$. Then
$$d(x, y) = \frac{\widetilde d_n(x_n, \widetilde y)}{2^n} < \frac{\eta}{2^n} =: \varepsilon\,\,\,\, \textit{and} \,\,\,\,
d(g.x, g.y) = \frac{\widetilde d_n(g_n.x_n, g_n.\widetilde y)}{2^n} \geq \frac{\sigma}{2^{n + 1}} =: \delta.$$
Hence $\diam(g.D_\epsilon(x)) \geq \delta$.
Since $\eta$ is any arbitrarily small positive number, then $\varepsilon  = \frac{\eta}{2^n}$ is also any arbitrarily small positive number.
As $x$ is any point from $X,$ the proven means the sensitivity of the group $G=\prod_{i\in J} G_i$ to initial conditions.
\end{proof}

\begin{theorem}\label{Prod_sens2}
Let $G_i$, $i = \overline{1, m}$, be a homeomorphism group of a metric space $(X_i, d_i)$. 
Assume that $(X,d) = \prod_{i = 1}^m(X_i, d_i)$. Then the direct product of groups $G =\prod_{i = 1}^m G_i$ 
is sensitive to initial conditions on $(X, d)$ if and only if there exists $n$, $1 \leq n \leq m$, such that the group $G_n$ is sensitive to initial conditions.
\end{theorem}
\begin{proof}
Sufficiency is proved similarly to Theorem~\ref{Prod_sens}. Let's prove the necessity. To do this, assume that the product of groups 
$G = \prod_{i = 1}^m G_i$ is sensitive to initial conditions and at the same time each group $G_i$ is not sensitive to initial 
conditions. Therefore for every $\delta_i > 0$ there exist point $x_i \in X_i$ and $\varepsilon_i > 0$ such that for every $g_i \in G_i$ 
the following inequality $\diam(g_i.D_{\varepsilon_i}(x_i)) \leq \delta_i$ is satisfied. Put 
$\varepsilon = \min\{\varepsilon_i|\, i = \overline{1,m}\} \,\, \Rightarrow \,\, \varepsilon > 0$, $\delta_i = \frac{\delta}{2m}$, 
where $\delta$ is the sensitivity constant of the action of the group $G$. Let $x = \{x_i\}$. 
Then for every $y = \{y_i\} \in D_{\varepsilon}(x)$ we get 
$$d(g.x, g.y) = \sum_{i = 1}^m d_i(g_i.x_i, g_i.y_i) \leq \sum_{i = 1}^m \frac{\delta}{2m} = \frac{\delta}{2},$$ 
consequently $\diam(g.D_\varepsilon(x)) \leq \delta$ that contradicts the sensitivity of the group $G$ to initial conditions. 
Thus, there exists $n$, $1 \leq n \leq m$, such that the group $G_n$ is sensitive to initial conditions.
\end{proof}

\begin{corollary}\label{CProd_sens}
Let $(X,d) = \prod_{i = 1}^m (X_i, d_i)$, where $(X_i, d_i)$ are metric spaces. Let $G_i$ be a homeomorphism group of $X_i$. 
Then the expansiveness of one of the groups $G_i$ entails the sensitivity of the product group $G =\prod_{i = 1}^m G_i$ to initial conditions.
\end{corollary}

\begin{example}\label{E2} If the group of homeomorphisms $G_1$ is expansive on $(X_1, d_1)$, and for $i = \overline{2,m}$ the group $G_i$ is
an isometry group of a metric space $(X_i, d_i)$, then, according to Corollary~\ref{CProd_sens}, the 
group $G = \prod_{i = 1}^m G_i$ is sensitive to initial conditions on the product $(X, d) = \prod_{i = 1}^m (X_i, d_i).$
\end{example}

\section{Chaotic actions of groups}\label{S7}

\subsection{Proof of Theorem~\ref{ThC1} }
Let $A$ be an arbitrary set of indexes. Suppose that the group $G = \prod_{\alpha \in A} G_\alpha$ acts chaotically on the Tychonoff product
$X = \prod_{\alpha \in A} X_\alpha$ of topological spaces $X_\alpha$. By Definition~\ref{Chaos2}, $G$ has a dense non-closed orbit and a dense union of 
closed orbits. According to Theorem~\ref{Trans1}, the group $G$ has a dense non-closed orbit in $X$ if and only if for every 
$\alpha \in A$ the group $G_\alpha$ has a dense orbit in $X_\alpha$. By Theorem~\ref{Closed}, $G$ has a dense union of closed orbits if 
and only if the union of closed orbits of the group $G_\alpha$ are dense on $X_\alpha$ for every $\alpha \in A$. Thus, a chaotic behavior of 
the group $G$ on $X$ is equivalent to a chaotic behavior of each group $G_\alpha$ on $X_\alpha$ for every $\alpha \in A$.
\qed

\vspace{3mm} Further we use notations introduced in Section~\ref{S5}.

\begin{theorem}\label{T1}
For every subset $J \subset \mathbb N$, let $G_i$, $i \in J$, be a homeomorphism group of metric space $(X_i, \widetilde d_i)$ and on 
the product of metric spaces $(X,d) = \prod_{i\in J}(X_i, \widetilde{d}_i)$ the direct product of groups $G = \prod_{i\in J}G_i$ is given. 
Assume that $(X, d)$ is a locally compact metric Baire space. If $G$ acts on $X$ chaotically, then:
\begin{enumerate}
\item[(1)] the group $G$ is sensitive to initial conditions;
\item[(2)] for every $i \in J$ the group $G_i$ is chaotic and sensitive.
\end{enumerate}
\end{theorem}
\begin{proof} First note that by Theorem~\ref{ThC1}, every group $G_i$, $i \in J$, is chaotic. Since $(X, d)$ is a locally compact metric Baire 
space, then every factor $(X_i, \widetilde{d}_i)$, $i \in J$, is also a locally compact metric Baire space. Hence, according to Theorem~\ref{Sens1}, chaoticity of
$G$ implies its sensitivity to initial conditions on $X$ and chaoticity of $G_i$ implies the sensitivity of $G_i$ on $(X_i, \widetilde{d}_i)$. 
\end{proof}

\begin{theorem}\label{T2}
For every subset $J \subset \mathbb N$, let $G_i$, $i \in J$, be a chaotic homeomorphism group of metric space $(X_i, \widetilde d_i)$ and on 
the product of metric spaces $(X,d) = \prod_{i\in J}(X_i, \widetilde{d}_i)$ the direct product of groups $G = \prod_{i\in J}G_i$ is given. 
If there exists $n\in J$ such that $(X_n, \widetilde d_n)$ is a locally compact metric Baire space, then the group $G$ is chaotic and sensitive 
to initial conditions.
\end{theorem}
\begin{proof} Theorem~\ref{ThC1} implies chaoticity of the group $G$. By the condition, $(X_n, \widetilde d_n)$ is a locally compact metric Baire space, 
then according to Theorem~\ref{Sensation2}, the group $G_n$ is sensitive to initial conditions. The application of Theorem~\ref{Prod_sens} 
allows us to state that the $G$ is also sensitive to initial conditions.
\end{proof}

\subsection{Chaotic products of countable homeomorphism groups of compact metrizable spaces}\label{S7_2}

\begin{lemma}\label{LNon}
Let $H$ be a countable group of homeomorphisms of a Polish space $X$. Then:
\begin{enumerate}
\item[(1)] every closed orbit of $H$ is discrete;

\item[(2)] if $X$ is compact, then every closed orbit of $H$ is finite.
\end{enumerate}
\end{lemma}
\begin{proof} Assume that $H.x$ is a closed orbit of $x\in X$. Then $H.x$ is a Polish space as a closed subset of 
a Polish space $X$. Show that the induced topology on the orbit $H.x$ is discrete. Assume that it is not true, hence $H.x$ has a non-isolated point
$x_0$. Since $H$ is countable, then $H$ is represented in the form $H = \{h_i\,|\, i\in\mathbb N\}$. As $x_0$ is non-isolated, there exists a
sequence $z_n = g_n.x$, $g_n\in H$, such that $z_n\to x_0$ as $n\to +\infty.$ Now check that every point of the orbit $H.x$ is non-isolated. Pick
$y\in H.x,$ then there is an element $g\in H$ for which $y = g.x_0.$ Hence $y_n = g.z_n\to y$ as $n\to +\infty,$ i.e. $y$ is non-isolated.

As $h_i.x$ is not isolated in $H.x$, the set $U_i = H.x\setminus h_i.x$ is dense and open in $H.x$. Since $H.x$ is a Polish space, then $H.x$ 
is also a Baire space. According to the definition of Baire space, the intersection $\bigcap_{i\in\mathbb N}U_i$ is dense in $H.x.$ As 
$\bigcap_{i\in\mathbb N}U_i = \varnothing,$ we get a contradiction. Hence every point of $H.x$ is isolated. Thus, the statement $(1)$ is proved.

Suppose now that a Polish space $X$ is compact. In this case every closed orbit $H.x$, $x\in X$, of $H$ is finite 
as a discrete compact subspace of the compact $X$. Therefore, the statement $(2)$ is also true.
\end{proof} 

\vspace{3mm}
Recall that a group $G$ is referred to as {\it residually finite} or {\it finitely approximable}, if for every non-neutral element 
$g\in G$ there exists a normal subgroup, not containing $g$, of finite index in $G$.

\begin{theorem}\label{ThC3}
If a countable group of homeomorphisms $G$ is chaotic (in the sense of Definition~\ref{Chaos2})
on a metrizable compact space $X$, then $G$ is residually finite and sensitive to initial conditions.
\end{theorem}
\begin{proof} Lemma~\ref{LNon} and Proposition~\ref{Birkhoff} imply that definitions of chaos \ref{Chaos1} and \ref{Chaos2} are equivalent. 
Hence the homeomorphism group $G$ is chaotic in the sense of Definition~\ref{Chaos1}. It follows from \cite[Th. 1]{CarDEKP} that every chaotic 
group of homeomorphisms is residually finite. According to Theorem~\ref{Sens1}, chaoticity of the group $G$ in the sense of Definition~\ref{Chaos2}
implies sensitivity of $G$ to initial conditions.
\end{proof}

As it is known, the following groups are residually finite:
\begin{enumerate}
\item[(1)] matrix groups $SL(n, \mathbb Z)$ for all $n \geq 2;$
\item[(2)] finitely generated linear groups;
\item[(3)] (finite or infinite) direct products of residually finite groups;
\item[(4)] countable generated free groups;
\item[(5)] finitely generated nilpotent groups;
\item[(6)] quotients of residually finite groups by finite normal subgroups;
\item[(7)] fundamental groups of compact 3-manifolds
\end{enumerate}
and some others.

Groups having infinite simple subgroups and, in particular, simple groups are not residually finite \cite{CarDEKP}. 

\begin{remark}\label{R} In \cite{CarDEKP} the sensitivity to initial conditions is not investigated.
\end{remark}

\subsection*{Proof of Theorem~\ref{TC2}}
Note that every compact metrizable space is Polish. As it is well known, unlike the Baire space, any Polish space is either countable or has 
the continuum cardinality.

Assume that every group $G_i$, $i\in \mathbb N$, is chaotic. According to Theorem~\ref{ThC1}, the canonical action of the direct product 
of groups $G = \prod_{i\in \mathbb N} G_i$ is chaotic, i.e. the statement $(1)$ is true.

As the group $G_i$ is chaotic, the union of closed orbits of $G_i$ is dense in $X_i.$ Due to countability of $G_i$ and compactness of 
$X_i,$ by Lemma~\ref{LNon}, every closed orbit of $G_i$ is finite. Show that the union $F_i \subset X_i$ of closed orbits of
the group $G_i$, containing greater than one point, is dense in $X_i$. Otherwise there exists an open subset $U \subset X_i$ such that 
$U \cap F_i = \varnothing$. Since $G_i$ is chaotic, then the union of its closed orbits is dense in $X_i$. Therefore the union of one-point 
orbits of $G_i$ is dense in $U$. Continuation of $G_i$ implies $G_i|_U = \id_U$ that contradicts the existence of a dense orbit, i.e.
the chaoticity of $G_i$. Thus, $F_i$ is dense in $X_i$. Analogously to the proof Theorem~\ref{Closed}, we get that the subset 
$F = \prod_{i \in \mathbb N} F_i$ of $X =\prod_{i \in \mathbb N} X_i$ is a union of closed orbits of the group $G = \prod_{i \in \mathbb N} G_i$, 
and $F$ is dense in $X$. Every orbit $G.x$ where $x = \{x_i\} \in X$, $x_i \in F_i$, has continuum cardinality as a product of countable set of 
finite orbits $G_i.x_i$, $i \in \mathbb N$. By the Tychonoff theorem, the orbit $G.x$ is compact as the product of compacts orbits $G_i.x_i$.

Assume that there exists a point $x_0 \in G.x$ which is isolated in $G.x$. Now check that every point of the orbit $G.x$ is isolated. 
Otherwise there is non-isolated $y = g.x_0$ where $g\in G$. Hence there exists a sequence $y_n \to y$ as $n\to +\infty.$ Then 
$z_n = g^{-1}.y_n \to x_0$ as $n\to +\infty$, that contradicts to our assumption. Therefore orbit $G.x$ has no isolated points, i.e. $G.x$ is 
perfect subset of $X$. Consequently, the statement $(2)$ is proved.

According Definition~\ref{Chaos2}, every chaotic group $G_i$ has a dense orbit $G_i.v_i,$ $v_i\in X_i.$ Note that $G_i.v_i$ is a countable subset
of $X_i,$ hence the orbit $G.v = \prod_{i\in\mathbb N}G_i.v_i$ has continuum cardinality. Since 
$\overline{G.v} = \prod_{i\in\mathbb N}\overline{G_i.v_i} = \prod_{i\in\mathbb N}X_i = X$, the orbit $G.v$ is dense in $X$. 
Hence the statement $(3)$ is proved.

The statements $(4) - (5)$ are corollaries of Theorem~\ref{ThC3}.

Let for every $i\in\mathbb N$, the group $G_i$ has a fixed point $x_i^0$. Let $Y$ be the union of all finite orbits of the group 
$G = \prod_{i \in \mathbb N} G_i$. Note that $x_0 = \{x_i^0\}$ is a fixed point of $G$, hence $x_0 \in Y$. Observe that $z = \{z_i\} \in Y$ 
if and only if there exists a finite subset $A \subset \mathbb N$ such that $z_i = x_i^0$ for $i\in\mathbb N\setminus A$ and $z_i$ 
has a finite orbit $G_i.z_i$ for $i\in A$. As the intersection of $Y$ with every set from
the canonical base of the Tychonoff topology of $X$ is a nonempty set, $Y$ is dense in $X$. Thus the statement $(6)$ is proved.
\qed

\section{Products of groups generated by generalized horseshoe maps}\label{S8}
\subsection{Bi-infinite sequence of symbols}
Let $S = \{1, 2, ... , N\}$ with $N \geq 2$. Equip the set $S$ with discrete metric $d$. Note that the metric topology coincides with the
discrete topology on $S$. Let $\{S_i|\,i \in \mathbb Z\}$ be a family of topological spaces, and $S_i = S,$ $d_i = d$ for every $i \in \mathbb Z$. 
Let $\Sigma^N$ be a Tychonoff product of this family: $\Sigma^N = \mathop \prod \limits_{i \in \mathbb Z} S_i$. 
Every point of $\Sigma^N$ can be represented as a bi-infinite sequence of $N$ symbols: 
$$\sigma \in \Sigma^N \, \Leftrightarrow\, \sigma = (... , \sigma_{-k}, ... , \sigma_{-1}, \sigma_0, \sigma_1, ... , \sigma_k, ...),$$ where 
$\sigma_i \in \{1, 2, ... , N\}$ for every $i \in \mathbb Z$. The topological space $\Sigma^N$ is compact (as product of compact spaces), 
totally disconnected (as product of totally disconnected spaces), perfect and uncountable. Also $\Sigma^N $ is a metrizable space. 
The metric on $\Sigma^N$ can be defined by the following way 
$$\rho(\sigma, \tau) = \mathop \sum \limits_{i \in \mathbb Z} \frac{1}{2^{|i|}} \frac{d_i(\sigma_i, \tau_i)}{1 + d_i(\sigma_i, \tau_i)},$$ 
where $d_i$ is the metric on the space $S_i$. The topology induced by the metric $\rho$ on $\Sigma^N$ is the same as the Tychonoff topology 
on the product $\prod_{i \in \mathbb Z} S_i$.

As is known, the space $\Sigma^2$ is homeomorphic to the standard ternary Cantor set \cite[Example 3.1.28]{Eng}. 

Now we define a map $g: \Sigma^N \to \Sigma^N$ as follows: $(g(\sigma))_i = \sigma_{i + 1}$. The map $g$ is referred to 
as {\it the full $N$-shift}. Consider a homeomorphism group $G = \langle g \rangle$ generated by $g$. 

\begin{proposition}\label{shift_is_chaotic} The homeomorphism group $G = \langle g \rangle$ is chaotic on the space $\Sigma^N$, and the 
union of finite orbits is infinite countable and dense in $\Sigma^N$. Besides, the group $G$ is sensitive to initial conditions.
\end{proposition}
\begin{proof}
According to \cite[Prop. 3.9.4]{Sturman-Ottino-Wiggins}, the group $G$ has a dense orbit.

Let $\sigma = (\sigma_i) \in \Sigma^N$. It is sufficiently to show that for every $\varepsilon > 0$ there is periodic point 
$\tau \in D_\varepsilon(\sigma)$ of the group $G$. Let $m > \log_2 \frac{1}{\varepsilon}$ and $m \in \mathbb N$. The point 
$\tau = \{\tau_i\}$, where $\tau_{k(2m + 1) + j} = \sigma_j\,\,\, \forall k \in \mathbb Z,\,\,\, \forall j\in [-m, m]\cap\mathbb Z$, 
is periodic with the minimal period $2m + 1.$ Besides, $\tau \in D_\varepsilon(\sigma)$ because 
$\rho(\sigma, \tau) \leq \mathop \sum \limits_{i = m + 1}^\infty \frac{1}{2^i} = \frac{1}{2^m} < \varepsilon$.
Thus, the set of periodic orbits is dense in $\Sigma^N$.

By analogy with \cite[Corollary 2.5.1]{Katok}, one can proved that the set of periodic orbits is infinite countable.

Moreover, the topological space $\Sigma^N$ has a countable base, then it is a compact Polish space, hence Theorem~\ref{Sens1} implies 
that the chaotic group of homeomorphisms $G$ of $\Sigma^N$ is sensitive to initial conditions.
\end{proof}

\begin{remark} \label{sequence_prod}
Let $J$ be any subset of $\mathbb N$ and $\Sigma^N_i = \Sigma^N$ for all $i\in J.$ Applying \cite[Prop.~2.3.7]{Eng}, we get that 
the Tychonoff product $\Sigma = \prod_{i\in J}\Sigma^N_i$ is homeomorphic to the space $\Sigma^N$.
\end{remark}

\subsection{$N$-ary Cantor sets}\label{Cantor2}
Here we recall the known generalization of the standard ternary Cantor set and the relation between this generalization and 
the space $\Sigma^N$.

\vspace{2mm}
Let $N \in \mathbb N$, $N > 2$ and $N$ is odd. The $N$-ary Cantor set can be constructed in the similar way as the standard ternary Cantor 
set. Let $C_0 = [0, 1]$, $C_1 = [0, \frac{1}{N}] \cup [\frac{2}{N}, \frac{3}{N}] \cup ... \cup [\frac{N - 1}{N}, \frac{N}{N}]$. Further 
we subdivide every closed interval of the set $C_1$ 
into $N$ equal parts and delete open intervals with even numbers. The remaining set is denoted $C_2$. By repeating this process, we get 
sets $C_i$ for every $i \geq 2$. 
The set $C = \bigcap_{i \in \mathbb N} C_i$ is the $N$-ary Cantor set \cite[section 2.3]{Chovanec}. In Figure~\ref{Cantor_set} it is illustrated this 
definition for case $N = 5$.
\begin{figure}[H]
\centerline{\includegraphics[width=0.6\columnwidth]{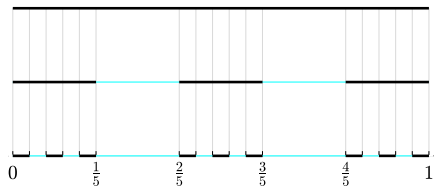}}
\caption{Three steps of construction of the $5$-ary Cantor set.}
\label{Cantor_set}
\end{figure}

Another equivalent approach to the definition of the $N$-ary Cantor set is known. Let $N \in \mathbb N$, $N > 2$ and $N$ is odd. Recall that the set 
$$\left\{\sum_{i \in \mathbb N} \frac{a_i}{N^i}|\, a_i \in \mathbb Z,\, 0 \leq a_i < N,\, a_i \, \text{is even}\right\}$$ is called the $N$-ary Cantor set 
\cite[section 1.9.1]{A_Kh_Sh}.

The $N$-ary Cantor set is homeomorphic to the space $\Sigma^{\frac{N + 1}{2}}$. The proof of this fact is fully similar the proof that the standard ternary 
Cantor set is homeomorphic to the space $\Sigma^2$ \cite[Example 3.1.28]{Eng}.

\subsection{Generalized horseshoe maps}
Original notion of the horseshoe map belongs to S. Smale. As is known, the invariant set $\Lambda$ of the horseshoe map is homeomorphic to the standard 
ternary Cantor set and the restriction of the horseshoe map to the set $\Lambda$ is conjugate to the $2$-shift map on the space $\Sigma^2$ \cite[Section 2.5]{Katok}.

The definition of generalized horseshoe map of length $N$ can be found in \cite[Definition 2.7]{Hirasawa-Kin}. In Figure~\ref{Gen_horseshoe} it is 
illustrated this definition for case $N = 3$.

\vspace{10mm} 
\begin{figure}[H]
\centering
\vspace*{-5em}
\subfloat[]
{
\includegraphics[width=0.7\columnwidth]{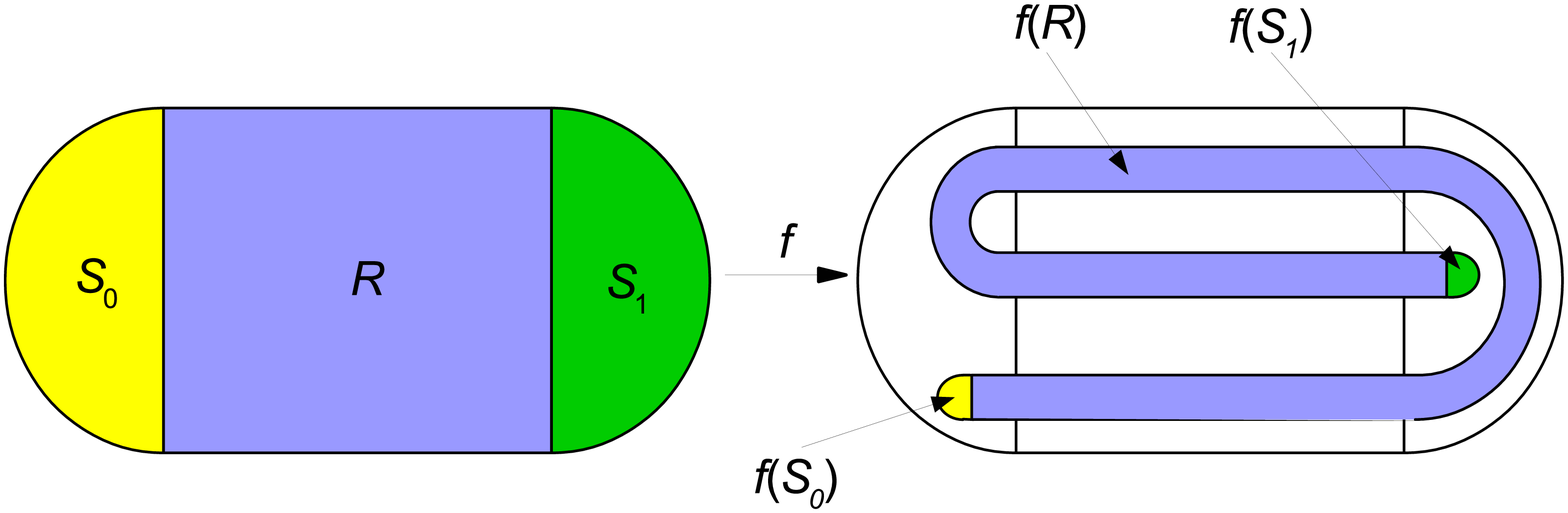}
\label{Gen_horseshoe_a}
}\\
\vspace*{-3em}
\subfloat[]
{
\includegraphics[width=0.7\columnwidth]{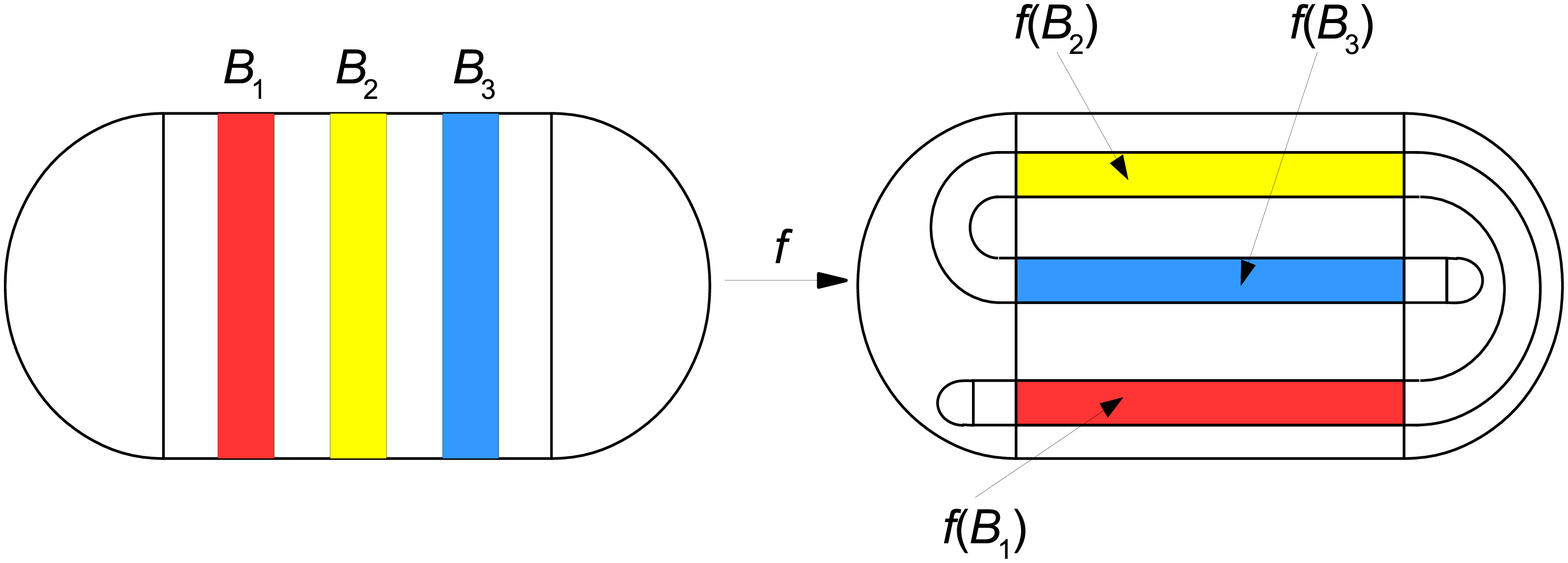}
}
\caption{The generalized horseshoe map of length 3.}
\label{Gen_horseshoe}
\end{figure}

Let $D = S_0 \cup R \cup S_1$. The generalized horseshoe map $f$ of length $3$ has the set $D$ as its domain. The image of the set 
$D$ respectively $f$ is shown in Figure~\ref{Gen_horseshoe_a}.

The invariant set $\Lambda$ of the generalized horseshoe map of length $N$ is homeomorphic to the $(2N - 1)$-ary Cantor set and, consequently, 
to the space $\Sigma^N$. The full $N$-shift is conjugate to the restriction of the generalized horseshoe map of length $N$ to the invariant set 
$\Lambda$ \cite[Section 4.2.3]{Sturman-Ottino-Wiggins}, \cite[Remark 2.9]{Hirasawa-Kin}.

\subsection{Chaotic groups of homeomorphisms on countable products of Cantor sets}
Let $A$ be any index set. Consider a family of chaotic groups $\{G_\alpha = \langle g_\alpha \rangle\}_{\alpha \in A},$ where $g_\alpha$ is 
the full $N_\alpha$-shift of $\Sigma^{N_\alpha}$. According to Theorem~\ref{ThC1}, the product of groups $G = \prod_{\alpha \in A}G_\alpha$ 
of homeomorphisms of the space $\Sigma = \prod_{\alpha \in A}\Sigma^{N_\alpha}_\alpha$ is chaotic.

Let the set $A$ be countable. As we have seen above, every $\Sigma^{N_\alpha}$ is a compact metric space, hence $\Sigma^{N_\alpha}$ is 
a compact metric Polish space. Consequently, $\Sigma$ is a compact metric Polish space, so according to Theorem~\ref{Sens1}, $G$ is sensitive 
to initial conditions on $\Sigma$.

In the case when $N_\alpha = N$ for every $\alpha \in A \subset \mathbb N$, we get $\Sigma^{N}_\alpha := \Sigma^{N_\alpha} = \Sigma^{N}$  and 
$\Sigma = \prod_{\alpha \in A}\Sigma^N_\alpha$. According to Remark~\ref{sequence_prod}, $\Sigma$ is homeomorphic to $\Sigma^{N}$. Thus we get 
new chaotic actions of free Abelian groups $G$ with countable set of generators on the space $\Sigma$ homeomorphic to $(2N - 1)$-ary Cantor set.

\section{Construction of chaotic actions of groups on topological manifolds}\label{S9}
\subsection{Chaotic actions of the group $\mathbb{Z}$ on every closed surface}\label{ssur}
 
By a closed surface we mean a connected compact topological two-dimensional ma\-{nifold} $M$ without boundary.
 
In the proof of Theorem~\ref{Surfaces} we use the concept of a two-dimensional orbifold. A simple exposition of the theory of compact two-dimensional
orbifolds can be found in \cite{Scott}. The necessary information about orbifolds used by us is contained in \cite{BZh}.

\begin{theorem}\label{Surfaces} For every closed surface $M$ there exists a countable family of chaotic groups of homeomorphisms, 
isomorphic to the group $\mathbb Z$, such that the union of finite orbits of every such group is dense in $M$.

\end{theorem}
\begin{proof} Show that for each $k, m\in\mathbb N$, $k, m\geq 3$, the following matrix \begin{equation}\label{5E}
A = A(k, m) = \left(
\begin{array}{ccc}
1&k\\
m&1 + km
\end{array}
\right)\,\,\, 
\end{equation} induces a chaotic homeomorphism $g_A$ of the closed two-dimensional disk $\mathbb B^2.$
Let $f_A$ be the Anosov torus automorphism defined by $A.$

We represent the torus $\mathbb T^2$ as the square $[-1/2, 1/2]\times[-1/2, 1/2]$ in a Cartesian coordinate system $Oxy$ 
on a plane $\mathbb R^2$ with identified opposite sides. In other words $\mathbb T^2$ has coordinates $(x, y)$ where $x$ and $y$ are
periodic of period one. For each fixed pair of numbers $k, m\in\mathbb N$, $k, m\geq 3$, consider subsets $P = [-1/2, 1/2]\times[-1/k, 1/k]$ 
and $Q = [-1/m, 1/m]\times[-1/2, 1/2]$ of $\mathbb T^2$. Note that $P$ and $Q$ are two overlapping annuli on the torus $\mathbb T^2$. 
Denote the union of the annuli by $R = P \cup Q$ and the intersection by $S = P \cap Q$, see Figure~\ref{Pillow}. Define maps $f: R\to R$ 
and $h: R\to R$ by the following equations:
\begin{equation}\label{6E}
f(x, y, k) = \left\{
\begin{array}{ll}
(x + ky, y) & \textrm{if}\,\, (x, y)\in P, \\
(x, y) & \textrm{if}\,\, (x, y)\in R\setminus P,
\end{array}
\right.
\end{equation} 

\begin{equation}\label{7E}
h(x, y, m) = \left\{
\begin{array}{ll}
(x, y + mx) & \textrm{if}\,\, (x, y)\in Q, \\
(x, y) & \textrm{if}\,\, (x, y)\in R\setminus Q.
\end{array}
\right.
\end{equation} 

The map $g = h\circ f: R\to R$ is called the {\it toral linked twist map} defined by composing $f$ and $h$.
According to \cite[Theorem A]{Dev}, the constructed linked twist map $g: R\to R$ is topologically mixing and 
the periodic points of $g$ are dense in $R$. Therefore, the homeomorphism group $G = \langle g \rangle$ generated by $g$ is 
topologically transitive and chaotic on the topological space $R.$ Note that $g = g(k, m)$ and $A = A(k, m)$ are related by the equality
$g(x, y) = A \left(
\begin{array}{ccc}
x\\
y\end{array}
\right)$ in some neighborhood $U$ of a point $(0, 0)$, $U\subset S\subset\mathbb T^2.$ 

\begin{figure}[H]
\centerline{\includegraphics[width=0.65\columnwidth]{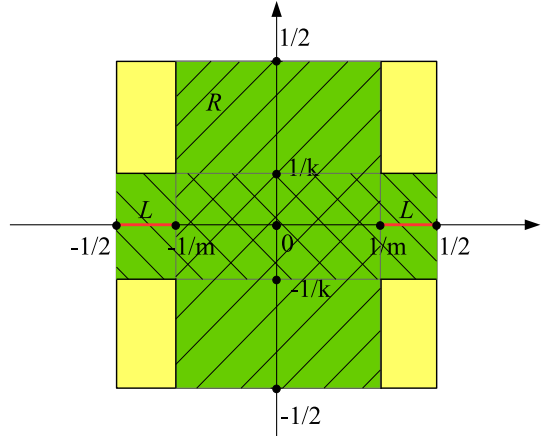}}
\caption{Representation of the torus on the unit square.}
\label{Pillow}
\end{figure}

Let $E_2$ be the unit two-dimensional matrix. Denote by $\varphi$ the homeomorphism of the torus $\mathbb T^2$ given by the matrix $-E_2.$ 
Then $\Phi = \left\langle\varphi\right\rangle$ is a homeomorphism group of $\mathbb T^2.$ Identify the orbit space 
${\cal N} = \mathbb T^2/\Phi$ with the rectangle $[-1/2, 1/2]\times[0, 1/2]$ the sides of which are identified in the way indicated 
by the arrows in Figure~\ref{Pillow2_a}. Denote by $p: \mathbb T^2\to\cal N$ the quotient map. Note that $\cal N$ is the orbifold "Pillow"$\,$ which is 
homeomorphic to the standard sphere $\mathbb S^2$, and both subsets $p(R)$, colored green, and $p(\mathbb T^2\setminus R)$, colored yellow, are 
homeomorphic to a closed disk $\mathbb B^2$ (see Figure~\ref{Pillow2_b}). Identify the topological space of $\cal N$ with the sphere $\mathbb S^2$.
\begin{figure}[H]
\centering
\vspace*{-2em}
\subfloat[Representation of the orbifold "Pillow".]
{
\includegraphics[width=0.4\columnwidth]{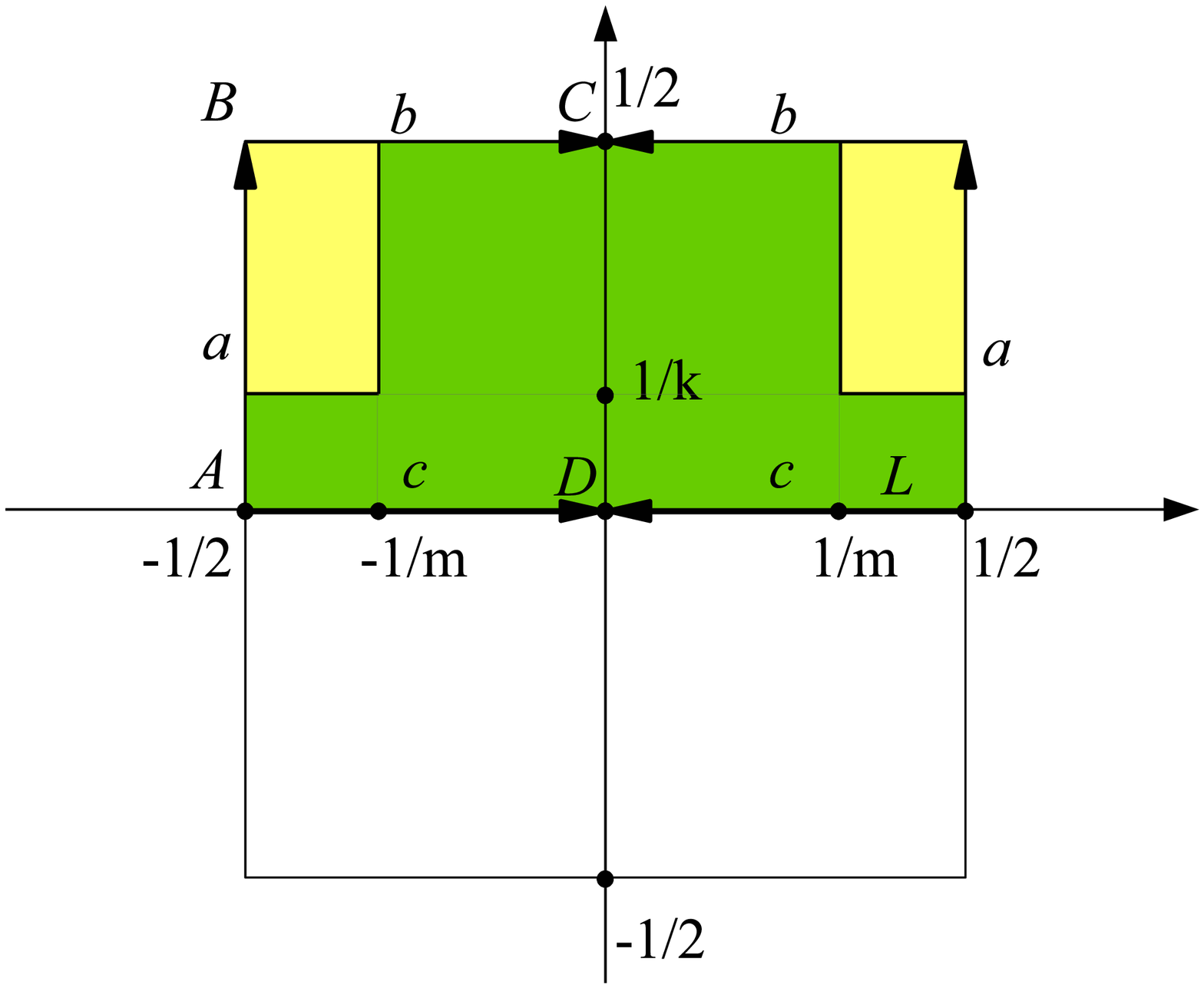}
\label{Pillow2_a}
}
\hspace{3mm}
\subfloat[The orbifold "Pillow" is homeomorphic to the standard sphere.]
{
\includegraphics[width=0.4\columnwidth]{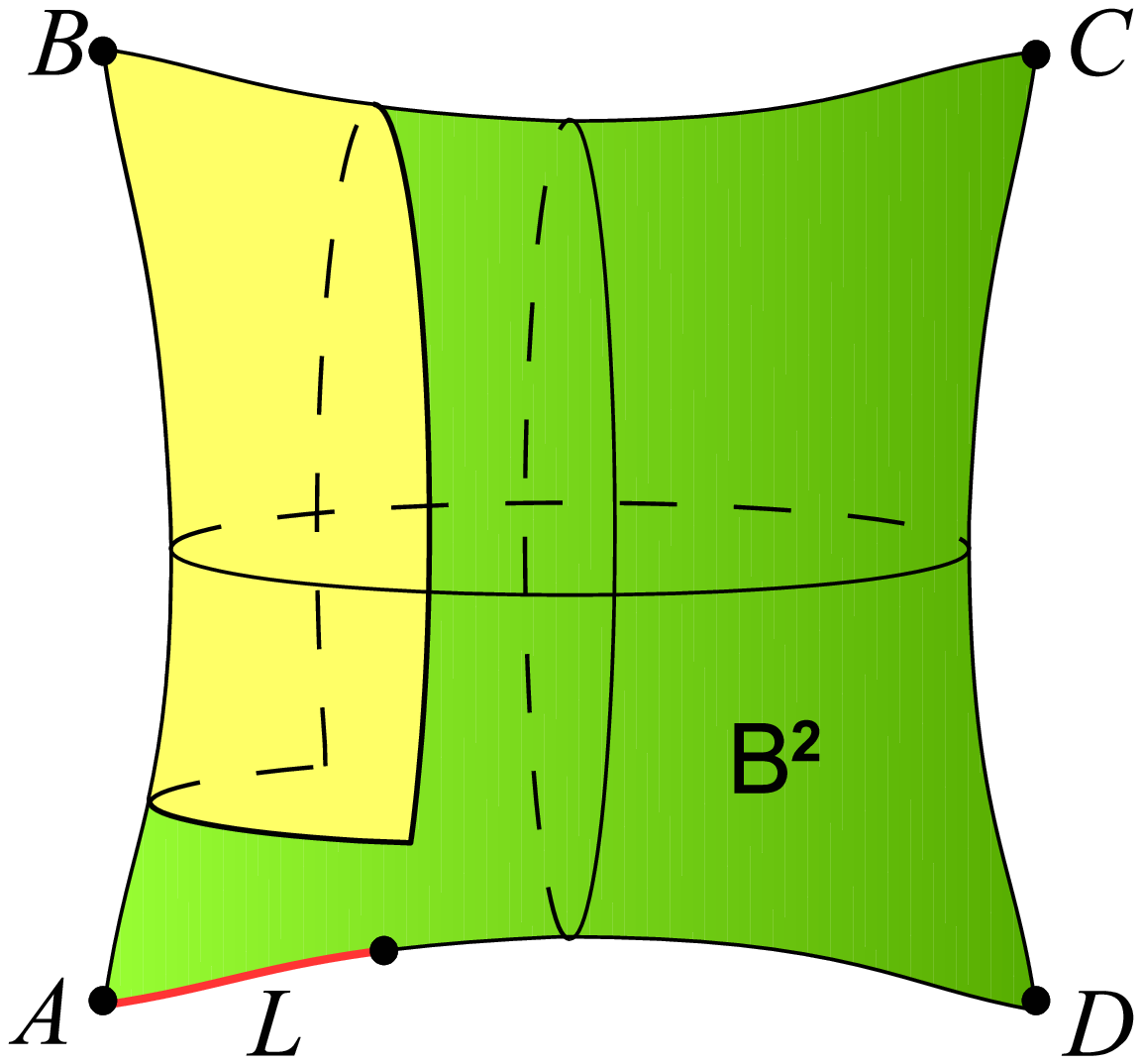}
\label{Pillow2_b}
}
\caption{Orbifold "Pillow".}
\label{Pillow2}
\end{figure}

It is easy to check that  $g|_{\partial R} = \id|_{\partial R}$ where $g = g(k, m)$ is defined by the matrix $A = A(k, m)$. 
Therefore we can continue $g$ for the entire torus $\mathbb T^2$ such that 
$g|_{\mathbb T^2\setminus R} = \id|_{\mathbb T^2\setminus R}.$ We denote the resulting torus homeomorphism also by $g.$ Emphasize that 
the points colored yellow in Figure~\ref{Pillow}, are fixed relative to the homeomorphism $g.$ Since the map $g$ satisfies the equality 
$g(-x,-y) = - g(x, y)$ for every $(x,y)\in\mathbb T^2$, it induces a homeomorphism $g_{\mathbb S^2}: \mathbb S^2\to\mathbb S^2$ satisfying the 
following equality $p\circ g = g_{\mathbb S^2}\circ p.$ Therefore the restriction $g_{\mathbb B^2}$ of $g_{\mathbb S^2}$ to the closed disk 
$\mathbb B^2 = p(R)$ (colored green in Figure~\ref{Pillow2_b}) satisfies the commutative diagram
\begin{equation}\label{EC}
\begin{CD}
R @>{g}>> R \\
@V{p}VV @VV{p}V \\
\mathbb B^2 @>g_{\mathbb B^2}>> \mathbb B^2,
\end{CD} 
\end{equation}
and the projection $p: R\to\mathbb B^2$ is a surjective continuous open map. As a finite orbit of the group $G$ maps onto a finite orbit of the group 
$\Gamma = \langle g_{\mathbb B^2} \rangle$ with respect to $p: R\to\mathbb B^2$, the union of finite orbits of $\Gamma$ is dense in $\mathbb B^2.$ 
A dense orbit of $G$ maps onto a dense orbit of $\Gamma$. Thus, the group $\Gamma$ is chaotic in $\mathbb B^2$. Since $g_{\mathbb S^2}$ is equal 
to identity on another closed disc colored yellow, complementary to $\mathbb B^2$, the group $\Gamma$ fixes every point of the boundary 
$\partial \mathbb B^2.$ Therefore it is possible to glue the boundary $\partial \mathbb B^2$ in an arbitrary way and to obtain a new surface $M$, 
and as the result, we can get every closed surface $M$. Denote the corresponding quotient map by $k: \mathbb B^2\to M.$ Since 
$\Gamma|_{\partial \mathbb B^2} = \id_{\partial \mathbb B^2},$ then $\Gamma$ induces an isomorphic group of homeomorphisms $\widetilde{\Gamma}$ of 
the surface $M$, and the mapping $k: \mathbb B^2\to M$ is a topological semi-conjugation of the groups $\Gamma$ and $\widetilde{\Gamma}$. Emphasize 
that the interior $\Int(\mathbb B^2)$ is invariant respectively $\Gamma$, and the restriction $k|_{\Int(\mathbb B^2)}$ is a homeomorphism conjugating 
$\Gamma|_{\Int(\mathbb B^2)}$ with $\widetilde{\Gamma}|_{k(\Int(\mathbb B^2))}$. Therefore, the group $\widetilde{\Gamma}$ is chaotic on $M$. 
Further we write $\widetilde{\Gamma} = \widetilde{\Gamma}(k, m)$, since this group is defined by the matrix $A = A(k, m)$ of the form \eqref{5E}.

Thus we get a countable family of chaotic groups $$\{\widetilde{\Gamma} (k, m)\,| k, m\in\mathbb N, k, m \geq 3\}$$
on every closed surface $M$. For short, further we denote such families by $\{\widetilde{\Gamma} (k, m)\}$.

More specifically, we consider the following ways of gluing the disk boundary $\partial \mathbb B^2$.

{\it Case I.} Consider the boundary $\partial \mathbb B^2\cong\mathbb S^1$ as a $4q$-polygon whose sides are glued together according to the scheme
$a_{1}b_{1}a_{1}^{-1}b_{1}^{-1}\cdot\cdot\cdot a_{q}b_qa_{q}^{-1}b_{q}^{-1}$, $q\in\mathbb N$. The scheme is obtained as follows: the direction of 
movement along the sides of the polygon is selected, the sides are written out in a row, the glued sides are marked with one letter, the degree $-1$ means 
the direction of gluing opposite to the direction of movement along the side. As a result we get a closed surface homeomorphic to the sphere with $q$ handles. Denote this surface by $\mathbb S^2_q$. Thus we get a countable family of chaotic groups $\{\widetilde{\Gamma} (k, m)\}$ on $\mathbb S^2_{q}$
for every $q\in\mathbb N.$

{\it Case II.} Consider the boundary $\partial \mathbb B^2\cong\mathbb S^1$ as a $2q$-polygon whose sides are glued together according to the scheme
$a_{1}a_{1}\cdot\cdot\cdot a_{q}a_{q}$, $q\in\mathbb N$. In this case we get a closed non-orientable surface homeomorphic to the sphere with $q$ Mobius bands which will be denoted by $N\mathbb S^2_{q}$. Thus we get an infinite series of chaotic groups $\{\widetilde{\Gamma} (k, m)\}$ on each surface 
$N\mathbb S^2_q$, $q\in\mathbb N.$

{\it Case III.}  Let us glue the boundary $\partial \mathbb B^2\cong\mathbb S^1$ into an arbitrary segment denoted by $L_0$ such that as the result we get 
a topological sphere $\mathbb S^2$. In the same way as above we get an infinite series of chaotic groups $\{\widetilde{\Gamma} (k, m)\}$
on $\mathbb S^2$.

Since every closed surface $M$ is homeomorphic to one of the canonical surfaces $\mathbb S^2$, $\mathbb S^2_q$ or $N\mathbb S^2_q$ where $q\geq 1$, 
then in the above way we obtain a countable family of chaotic homeomorphism groups $\{\widetilde{\Gamma} (k, m)\cong\mathbb Z\}$ on $M$.

It is well known that finite orbits of every Anosov automorphism of the torus $\mathbb T^2$ form a countable dense subset in $\mathbb T^2$.
However, this is not true for the constructed homeomorphism $g$ in general. In the next section we will construct a homeomorphism $g$ 
having a continuum of fixed points. Hence, in this case, the set of all finite orbits of both groups $\Gamma$ and $\widetilde{\Gamma}$ 
has the cardinality of the continuum.
\end{proof}

\begin{remark} In contrast to \cite{CarDEKP}, we have constructed an infinite countable family of chaotic actions of the group $\mathbb Z$ 
on each closed surface.
\end{remark}

\subsection{Chaotic actions of the group $\mathbb{Z}$ on noncompact\\ two-dimensional manifolds}\label{noncompact}

Consider the toral linked twist map $g = g(k, m): R\to R$ constructed in the previous section. Recall that $R$ is the union of two annuli
$P \cup Q$ on the torus $\mathbb{T}^2.$ Let $p: R\to\mathbb{B}^2$ be the projection satisfying the commutative diagram \eqref{EC}, hence 
$p(\partial R) = \partial\mathbb{B}^2\cong\mathbb{S}^1.$ 

At first we will pick out some point $z\in\partial\mathbb{B}^2.$ Identify $\partial\mathbb{B}^2\setminus \{z\}$ with the real line
$\mathbb{R}^1 = \bigcup_{n=-\infty}^{+\infty}[n, n+1]$. Represent $\mathbb{B}^2\setminus \{z\}$ as an polygon without a point $z$, the 
boundary of which is divided into a countable set of pairwise glued sides. Consider the following gluing rules.

{\it Case IV.} $[n, n+1] \sim a_n b_n a_n^{-1} b_n^{-1}\,\,\,\, \forall n\in\mathbb Z$. As the result of this gluing, we will get a noncompact 
two-dimensional manifold $M$ without boundary, homeomorphic to a plane with a countable family of handles which we denote by $\mathbb R^2_{\infty}$. 
Emphasize that $M$ is homeomorphic to the Loch Ness monster (see Figure~\ref{Loch_ness_monster}).

{\it Case V.} $[n, n+1] \sim a_n a_n\,\,\,\, \forall n\in\mathbb Z$. As the result of this gluing, we will get a noncompact two-dimensional manifold $M$ without
boundary, homeomorphic to a plane with a countable family of Mobius bands which we denote by $N\mathbb R^2_{\infty}$.

{\it Case VI.} Let us glue $p(\partial R) = \partial\mathbb{B}^2\cong\mathbb{S}^1$ to a segment $L_0$ as in Case III. As the result we get a topological sphere 
$\mathbb S^2$. As above, let $k: \mathbb B^2 \to \mathbb S^2$ be the respective quotient map. The image $k(\partial \mathbb B^2)$ of $\partial \mathbb B^2$ we denote by $L_0 \subset \mathbb S^2$. Consider an arbitrary proper compact subset $K_0$ of $L_0$ and a point $z_0\in L_0\setminus K_0$. We get a noncompact two-dimensional manifold denoted by $\mathbb R^2(K_0) := \mathbb S^2 \setminus (K_0\cup\{z_0\})$ homeomorphic to a plane without $K_0$. In Figure~\ref{Cantor_tree} you can see the noncompact surface $\mathbb R^2(K_0)$ where $K_0$ is the standard Cantor set.

Now we consider the restriction of $g: R\to R$, constructed in the previous section, to a circle $C = \{(x, 0)\,|\, -1/2\leq x\leq 1/2\}.$
The image $g(C)$ is a closed curve on the torus $\mathbb{T}^2$. Observe that 
$$g(x, 0) = (x, 0)\,\,\,\, \forall x\in [-1/2, -1/m] \cup [1/m, 1/2].$$
Emphasize that $[-1/2, -1/m] \cup [1/m, 1/2]$ is a connected segment in $R \subset \mathbb T^2$. The set $L = p([-1/2, -1/m] \cup [1/m, 1/2])$ 
is a connected topological 
segment in the inside of the disk $\mathbb{B}^2$ highlighted in red in Figure~\ref{Pillow2_b}, and $g|_L = \id_L.$ Assume that the boundary $\partial\mathbb{B}^2$ of $\mathbb{B}^2$ 
is glued in one of the ways specified above in Cases I -- VI. The result is a manifold $M$ where $M$ is either one of arbitrary closed surfaces 
$\mathbb S^2$, $\mathbb S^2_q$, $N\mathbb S^2_q$ where $q\geq 1$, or one of noncompact surfaces $\mathbb R^2_{\infty}$, $N\mathbb R^2_{\infty}$, 
$\mathbb R^2(K_0)$. Let $k:\mathbb{B}^2\to M$ be the corresponding quotient map. As no points from $L$ were not glued together, we identify the image 
$k(L)$ of $L$ in $M$ with $L$ and consider $L$ as subset of $M$. This fact allows us to remove an arbitrary closed subset $K$ of $L$ from $L$ and to 
obtain a noncompact two-dimensional topological manifold $\widehat{M} = \widehat{M}(K) := M\setminus K.$ Emphasize that we do not exclude the 
possibility $K = \varnothing.$ The homeomorphism $\widetilde{g} = \widetilde{g}(k, m)$ induces a chaotic homeomorphism $\widehat{g} = \widehat{g}(k, m)$ 
of $\widehat{M}$, and the group $\widehat{\Gamma} = \langle \widehat{g} \rangle$ has a dense set periodic orbits in $\widehat{M}$.

Indicate some important classes of obtained manifolds $\widehat{M} = \widehat{M}(K).$ 

If $K$ is a finite subset of $L$, then $\widehat{M}$ is a noncompact manifold with a finite number of ends.

If $K$ is an infinite subset of $L$ with one limit point $z',$ then $z'\in K$ and we get a noncompact manifold 
$\widehat{M} = M\setminus K$ with a discrete countable set of ends.

When $K$ is the standard Cantor set on $L$, we get a noncompact manifold $\widehat{M}$ containing the Cantor set of ends.
It is possible that $\widehat{M}$ has other ends obtained by removing a compact subset $K_0 \subset L_0$ where $L_0 \subset k(\partial \mathbb B^2)$ similarly to Case~{VI}.

Emphasize that for every closed subset $K$, the set of handles or Mobius bands on $\widehat{M}$ may be countable.

In Case VI, for $M = \mathbb R^2(K_0)$ we get $\widehat{M} = \widehat{M}(K_0, K)$. When $K_0 = \varnothing$ and $K$ is the standard Cantor set, 
the surface $\widehat{M}(K_0, K)$ is homeomorphic to the surface in Figure~\ref{Cantor_tree}. 

On every such constructing noncompact topological two-dimensional manifold $\widehat{M}$ the homeomorphism $\widetilde{g} = \widetilde{g}(k, m)$ 
induces a chaotic homeomorphism $\widehat{g} = \widehat{g}(k, m)$, and the group $\widehat{\Gamma} = \langle \widehat{g} \rangle$ has a dense set 
periodic orbits in $\widehat{M}$.

Thus we get an infinite countable family of chaotic groups
$\{\widehat{\Gamma} (k, m) = \langle \widehat{g}(k, m)\rangle\}$ on each surface $\widehat{M}$ constructed above.

\subsection{Chaotic groups acting on closed $n$-dimensional manifold}\label{sk}
Let $G$ be any countably generated free group. In \cite{CarK} G. Cairns, A. Kolganova and A.~Nielsen showed that every compact triangulable manifold $M$ 
of an arbitrary dimension $n$ greater than one admits a faithful chaotic action of the group $G$. 

Emphasize that as follows from \cite{CarK}, the free group $\Gamma = \langle g_1, g_2\rangle$ may be implemented as a chaotic group of 
homeomorphisms of every such manifold $M$.

\subsection{Chaotic groups of homeomorphisms on products of manifolds}\label{sk1}

Since every open subset of a Polish space is also Polish, then all topological manifolds defined above in Sections~\ref{S9} are locally compact
Polish spaces, hence they satisfies conditions of Theorems~\ref{Sensation2} -- \ref{ThC1}. For an arbitrary index set $A$ we use the constructed above 
chaotic actions of homeomorphism groups $G_\alpha$, $\alpha\in A,$ on topological manifolds $M_\alpha$ as building blocks for constructions of 
chaotic canonical action of the product of groups $G = \prod_{\alpha\in A} G_\alpha$ on $M = \prod_{\alpha\in A} M_\alpha$. 

According to Theorem~\ref{TC2}, if $G_i$, $i\in J\subset\mathbb N,$ is a countable chaotic group of homeomorphisms of $n_i$-dimensional closed 
triangulable manifold $M_i$, $n_i \geq 2,$ then the canonical action of the direct product of groups $G = \prod_{i\in J}G_i$ is chaotic on 
the product $M = \prod_{i\in J}M_i$. Here $M$ is a infinite-dimensional topological manifold, if the set $J$ is infinite countable, otherwise $M$ is a 
finite-dimensional topological manifold. In both cases, by Theorem~\ref{TC2}, every groups $G$ and $G_i$ have sensitivity to initial conditions.
Emphasize that the dimension of $M_i$ may be an arbitrary greater than one. 

Moreover, in the case $J = \mathbb N,$ according to Theorem~\ref{TC2}, there exists a dense subset $F \subset M$ which is the union of compact continuum
orbits of the group $G$, and every such orbit is a perfect subset of $M$. Besides there exists a dense continuum orbit of $G$ in $M$. 

If, moreover, every $M_i$ is a two-dimensional manifold with a chaotic action of the group $G_i\cong\mathbb Z$ constructed in 
Sections~\ref{ssur} -- \ref{noncompact}, then every group $G_i$ has finite orbits and, in particular, a fixed point. Therefore, 
in this case, the group $G = \prod_{i\in\mathbb N}G_i$ on the infinite-dimensional topological manifold $M = \prod_{i\in J}M_i$ has 
a continuum set of finite orbits, and the union of finite orbits is dense in $M$. Therefore every groups $G_i$ and $G$ are chaotic in the sense of both 
Definition~\ref{Chaos1} and Definition~\ref{Chaos2}.

\begin{example} Consider the standard two-dimensional torus
$\mathbb{T}^2 = \mathbb{R}^2/\mathbb{Z}^2$. Anosov torus automorphism $\mathbb{T}^2,$ given by a matrix
$A \in SL(2, \mathbb Z),$ is denoted by $g_A$. As is well known \cite{ZhR}, every matrix \begin{equation}\label{An}
A=\left(
\begin{array}{ccc}
a&b\\
c&d
\end{array}
\right),\,\,\, 
\end{equation}
where $ad-bc = 1$ and $a+d > 2$, defines an  Anosov automorphism $g_A$ of the torus $\mathbb T^2$
preserving its orientation . The group $G = \langle g_A\rangle \cong\mathbb Z$ generated by $g_A$ acts chaotically
on $\mathbb T^2$. It is well known that there exists a countable set of finite orbits of $G$, and this set is dense in $\mathbb T^2.$
There exists a countable family ${\mathcal A} = \{A_k\,|\, k\in\mathbb N\}$ such matrices defining Anosov automorphisms.

Consider the infinite-dimensional torus $\mathbb T^\infty = \prod_{i\in\mathbb N}\mathbb T^2_i$ where $\mathbb T^2_i = \mathbb T^2$ for each $i\in\mathbb N.$
Let $g_{A_i}$, $A_i\in{\mathcal A},$ be an  Anosov automorphism on $\mathbb T^2_i$. Let $G_i = \langle g_{A_i}\rangle$ and $G = \prod_{i\in\mathbb N} G_i.$
According to Theorem~\ref{TC2}, we get a chaotic group of homeomorphisms $G$ of the torus $\mathbb T^\infty$. Taking into account that $A_i$ may be an arbitrary
matrix belonging to $\mathcal A,$ we see that different groups $G$ form a continuum set. Emphasize that $G$ also has continuum cardinality. As every group $G_i$ 
has a fixed point, by Theorem~\ref{TC2}, the union of finite orbits of group $G$ is dense in $\mathbb T^\infty$, and $G$ has a fixed point. Moreover, the 
union of its compact orbits of continuum cardinality is also dense in $\mathbb T^\infty$, and every such orbit is a perfect subset of $\mathbb T^\infty$. 
Note that every dense orbit has continuum cardinality. Besides, by Theorem~\ref{TC2}, the group $G$ is sensitive to initial conditions.

All groups $G_i$, $i \in \mathbb N$, and $G$ are chaotic in the sense of Definition~\ref{Chaos1} as well as in the sense of Definition~\ref{Chaos2}.
\end{example}

\section*{Acknowledgements}

This work is supported by Russian Science Foundation, project No~22-21-00304.


\begin{thebibliography}{99}

\bibitem{AAB} E. Akin, J. Auslander, and K. Berg, \emph{When is a transitive map chaotic?}, in \emph{Convergence in Ergodic Theory and Probability}, V. Bergelson, P. March and J. Rosenblatt, eds., Gruyter, Berlin, 1996, pp. 25--40.

\bibitem{A_Kh_Sh} S. Albeverio, A. Y. Khrennikov, and V. M. Shelkovich, \emph{Theory of p-adic distributions: linear and nonlinear models}, Cambridge University Press, New York, 2010.

\bibitem{Assaf} D. Assaf and S. Gadbois, \emph{Definition of chaos}, Amer. Math. Monthly, 99(9) (1992), pp. 865--865.

\bibitem{BBC} J. Banks, J. Brooks, G. Cairns, G. Davis, and P. Stacey, \emph{On Devaney's definition of chaos}, Amer. Math. Monthly, 99(4) (1992), pp. 332--334.

\bibitem{Bar} A. Barzanouni, M.S. Divandar, and E. Shah, \emph{On Properties of Expansive Group Actions}, Acta Mathematica Vietnamica, 44(4) (2019), pp. 923--934. 

\bibitem{BGZ} Y.V. Bazaikin, A.S. Galaev, and N.I. Zhukova, \emph{Chaos in Cartan foliations}, Chaos: An Interdisciplinary Journal of Nonlinear Science, 30 (2020), p. 103116.

\bibitem{BZh} E.V. Bogolepova and N.I. Zhukova, \emph{Anosov actions of isometry groups on Lorentzian 2-orbifolds}, Lobachevskii mathematical journal, 42(14) 
(2021), pp. 3324--3335.

\bibitem{CarDEKP} G. Cairns, G. Davis, D. Elton, A. Kolganova, and P. Perversi, \emph{Chaotic group actions}, Enseignement Mathematique, 41 (1995), pp. 123--133.

\bibitem{CarK} G. Cairns and A. Kolganova, \emph{Chaotic actions of free groups}, Nonlinearity, 9(4) (1996), pp. 1015--1021.

\bibitem{CairnsKolgN} G. Cairns, A. Kolganova, and A. Nielsen, \emph{Topological transitivity and mixing notions for group actions}, The Rocky Mountain Journal of Mathematics, 37(2) (2007), pp. 371--397.

\bibitem{Chovanec} F. Chovanec, \emph{Cantor sets}, Science and Military Journal, 5(1) (2010), pp. 5--10.

\bibitem{Devaney} R.L. Devaney, \emph{An introduction to chaotic dynamical systems}, Addisson Wesley, Redwood City, 1986.

\bibitem{Dev} R.L. Devaney, \emph{Linked twist mappings are almost Anosov}, In: Z. Nitecki, C. Robinson (eds). Global Theory of Dynamical
Systems. Lecture Notes in Mathematics, Springer, Berlin, Heidelberg, 819 (1980), pp. 121--145.

\bibitem{Deza} M.M. Deza and E. Deza, \emph{Encyclopedia of distances}, Springer, Berlin, Heidelberg, 2009.

\bibitem{Eng} R. Engelking, \emph{General topology}, Mir, Moscow, 1986.

\bibitem{Hirasawa-Kin} M. Hirasawa and E. Kin, \emph{Determination of generalized horseshoe maps inducing all link types}, Topology and its Applications, 
139(1--3) (2004), pp. 261--277.

\bibitem{KM} E. Kontorovich, M. Megrelishvili \emph{A note on sensitivity of semigroup actions}, Semigroup Forum, 76:1 (2008), pp. 133--141. 

\bibitem{Katok} A. Katok and B. Hasselblatt, \emph{Introduction to the modern theory of dynamical systems}, Cambridge university press, New York, 1997.

\bibitem{Kech} A. Kechris, \emph{Classical descriptive set theory}, Springer-Verlag, New York, 1995.

\bibitem{NSan} A.C. Naolekar and P. Sankaran, \emph{Chaotic group actions on manifolds}, Topology and its Applications, 107 (2000), pp. 233--243.

\bibitem{Polo} F. Polo, \emph{Sensitive  dependence on initial conditions and chaotic group actions}, Proc. Am. Math. Soc., 138(8) (2010),  pp. 2815--2826.

\bibitem{ZhR} E.A. Rogozhina and N.I. Zhukova, \emph{Classification of compact lorentzian 2-orbifolds with noncompact full isometry groups}, Siberian mathematical journal, 53(6) (2012), pp. 1037--1050.

\bibitem{Sturman-Ottino-Wiggins} R. Sturman, J. M. Ottino, and S. Wiggins, \emph{The mathematical foundations of mixing: the linked twist map as a paradigm in 
applications: micro to macro, fluids to solids}, Cambridge university press, New York, 2006.

\bibitem{Scott} P. Scott, \emph{The geometries of 3-manifolds}, Bull. London Math. Soc., 15 (1983), pp. 401--487.

\bibitem{ZhM} N. I. Zhukova, \emph{Minimal sets of Cartan foliations}, Proceedings of the Steklov Institute of Mathematics, 256 (2007), pp. 105--135.

\end{thebibliography}
\end{document}